%% file: Explicit_renaming_of_bound_variables.tex
\documentclass[11pt]{amsart}
\usepackage[cp1251]{inputenc}
\usepackage[english]{babel}
\usepackage{amssymb,amsmath,amscd,amsthm,latexsym}
\usepackage[matrix,arrow,curve]{xy}
\CompileMatrices
\usepackage{prooftree}
\usepackage{array}
\def\ruleone#1#2{\prooftree #1 \justifies #2 \endprooftree}
\def\ruletwo#1#2#3{\ruleone{#1\quad #2}{#3}}
\def\ruledot#1{\prooftree \proofdotseparation=1.2ex\proofdotnumber=3
               \leadsto #1 \justifies #1 \endprooftree}
\theoremstyle{plain}
\newtheorem{theorem}{Theorem}[section]
\newtheorem{corollary}[theorem]{Corollary}
\newtheorem{lemma}[theorem]{Lemma}

\theoremstyle{definition}
\newtheorem{definition}[theorem]{Definition}
\newtheorem{convention}[theorem]{Convention}

\theoremstyle{remark}
\newtheorem{example}[theorem]{Example}
\numberwithin{equation}{section}
\newcommand{\tri}{\triangleright}
\newcommand{\ci}{\circ}
\newcommand{\ri}{\rightsquigarrow}
\newcommand{\mi}{\,,\,}
\newcommand{\is}{\diagdown}
\newcommand{\arrow}{\rightarrow}
\newcommand{\la}{\langle}
\newcommand{\ra}{\rangle}
\newcommand{\twoh}{\arrow\arrow}
\newcommand{\alp}{\equiv_{\alpha}}

\newcommand{\x}{a}
\newcommand{\y}{b}
\newcommand{\z}{c}
\newcommand{\term}{Term\,}
\newcommand{\sub}{Subst\,}
\newcommand{\var}{Var\,}
\newcommand{\absterm}{Abs\,}
\newcommand{\appterm}{App\,}
\newcommand{\closterm}{Clos\,}
\newcommand{\compsub}{Comp\,}
\newcommand{\conssub}{Cons\,}
\newcommand{\varlist}{list\,}
\newcommand{\slist}{List\,}
\renewcommand{\pi}{\mathcal{W}}
\begin{document}
\title{Explicit renaming of bound variables}
\author{George Cherevichenko}
\input{abstract}
\maketitle

\input{introduction}
\input{section1}
\input{freevariables}
\input{redukcii1}
\input{subjectreduction}
\input{normalforms}
\input{lambda-sigma}
\input{alpha}
\input{confluence}
\input{SN}
\input{Post}
\input{note}

\input{bib}
\end{document}

%% file: abstract.tex
\begin{abstract}We present the lambda calculus $\lambda\pi$ with explicit
substitutions and named variables. The characteristic feature of
this calculus is as follows: renaming of bound variables when
performing substitutions is done using  special reductions and may
be delayed.
\end{abstract}

%% file: introduction.tex
\section{Introduction}
There is a gap between lambda calculi with explicit substitutions
using De Brujn indices and lambda calculi with explicit
substitutions using ordinary (named) variables. The first follow
 the spirit of category theory. The second attempt to reflect
the ``real way to work with bound variables". We clarify this with
an example. Simultaneous substitution  will be denoted by
$$[x_1/N_1\mi x_2/N_2\mi\ldots\mi x_k/N_k]$$ Let's call this substitution $s$.
Suppose the variable $x$ is different from all
$x_1,x_2,\ldots,x_k$. By $[s\mi x/N]$ denote the substitution
 $$[x_1/N_1\mi x_2/N_2\mi\ldots\mi x_k/N_k\mi x/N]$$
According
 to~\cite{Stoughton},
the substitution $s$ moves under a binder this way
$$(\lambda x.M)[s]\arrow\lambda y.(M[s\mi x/y])$$
where $y$ is a ``fresh" variable.  The similar reduction for
categorical combinators is
$$\Lambda( M)\ci s\arrow\Lambda(M\ci\la s\ci F\mi S\ra)$$
where $F$ denotes the first projection and $S$ denotes the second
projection. A significant difference is that in the latter case
the substitution $s$ is multiplied by the first projection. Abadi,
Cardelli, Curien, and Levy in \cite{Abadi} suggested to use the
substitution $\uparrow$, corresponding to the first projection,
together with named variables. They have obtained the equality
$$(\lambda x.M)[s]=\lambda x.(M[(x/
x)\cdot(s\,\ci\!\uparrow)])$$ We  rewrite this equality as
$$(\lambda x.M)[s]\arrow\lambda x.(M[s\,\ci\!\uparrow\mi x/
x])$$ Abadi, Cardelli, Curien, and Levy write ``In this notation,
intuitively,\\ $x[\uparrow]$ refers to $x$ after the first
binder." To clarify this point, consider some typed calculus with
contexts, where  contexts are finite lists of the form $x_1:A_1,
x_2:A_2,\ldots, x_k:A_k$, where $A_1,A_2,\ldots,A_k$ are types and
\emph{repetitions of variables are permitted.}  A judgement of the
form $\Gamma\vdash x:A$ means ``the rightmost occurrences of the
variable $x$ in the context $\Gamma$ has type $A$." For example,
the judgement $x:A\mi x:B\vdash x:B$ is true, but the judgement
$x:A\mi x:B\vdash x:A$ is not true. But the judgement $x:A\mi
x:B\vdash x[\uparrow]:A$ is true.  The crucial idea is this: if we
allow repetitions of identical variables as in $\lambda x.\lambda
x.M$, then we must allow repetitions in contexts too. In this way
we will obtain some lambda calculus with explicit
substitutions and named variables such that:\\
 (1) It is close to  the calculi of categorical combinators;\\
 (2) It is convenient to work;\\
 (3) Renaming of bound variables when
performing substitutions is done using  special reductions and may
be delayed.

Now we must introduce a convenient notation.  To give a definition
of free variables  it is much more convenient to use the notation
$[s]M$ than   $M[s]$. Substitutions should be on the same side
where contexts and binders are. Composition of substitutions also
will be written in the reverse order (we will write $q\ci s$ where
it was written $s\ci q$). For example, the
rewrite rule\\
$\begin{array}{ll} & M[s][q]\arrow M[s\ci q]
\end{array}$\\[3pt]
will now look like this\\
$\begin{array}{ll} & [q][s]M\arrow[q\ci s]M
\end{array}$\\
Now we can write far fewer parentheses. For example, $[s]\lambda
x.[q]\lambda y.M$ is uniquely deciphered as $[s](\lambda
x.([q](\lambda y.M)))$. I chose the notation $s\ci M$  instead of
$[s]M$, because $s\ci\lambda x.q\ci\lambda y.M$ is easy to read,
this notation is close to the notation of category theory, and we
can now use angle brackets to denote  ordered pairs and nothing
else ($id\circ M$ looks better than $\la id \ra M$).\\ After some
doubts I have replaced the symbol $\uparrow$ by $\pi$. We will
have to supply this symbol with a subscript, and
$\la\pi_x\ci\pi_y\mi\pi_z\ci z\is z\ra$ is much easier to read
than $\la\uparrow_x\ci\uparrow_y\mi\uparrow_z\ci z\is z\ra$. The
symbols $\pi_x$  correspond to $\mathcal{W}_x$ from~\cite{Kesner}
to some extent, but are not the same.

 The sets of untyped
terms and substitutions  are defined inductively as follows:
\begin{align*}
M,N::&= x \mid  MN \mid \lambda x. M \mid s\ci M \\
s,q::&= id \mid \pi \mid \langle s\mi  N\is x \rangle \mid s\ci q
\end{align*}
where the symbol $x$ denotes an arbitrary variable.

 The sets of typed
terms and substitutions are defined inductively as follows:
\begin{align*}
M,N::&= x \mid  MN \mid \lambda x^A. M \mid s\ci M \\
s,q::&= id \mid \pi \mid \langle s\mi  N\is x \rangle \mid s\ci q
\end{align*}
where $A$ is an arbitrary type.\\
A usual simultaneous substitution $$[x_1/N_1\mi
x_2/N_2\mi\ldots\mi x_k/N_k]$$ in the new notation looks like
$$\la\ldots\la\la id\mi N_1\is x_1\ra\mi N_2\is x_2\ra\mi\ldots
N_k\is x_k\ra$$ For brevity, we will write  $$\la id\mi N_1\is
x_1\mi N_2\is x_2\mi\ldots\mi N_k\is x_k\ra$$ But now any two (or
more) of the variables $x_1,\ldots,x_k$
 may
coincide (as in contexts).

A \emph{judgement} is an expression of the form $\Gamma\vdash M:A$
or of the form $\Gamma\vdash s\tri\Delta$, where $\Gamma$ and
$\Delta$ are contexts, $A$ is a type, $M$ is a term, and $s$ is a
substitution.
\begin{definition} (Typing rules).\label{type}\\
$
\begin{array}{lll}
(i)& \Gamma,x:A\vdash x:A &\\[5pt]
(ii)& \ruleone{\Gamma\vdash
x:A}{\Gamma,y:B\vdash x:A} & (x\not\equiv y) \\[15pt]
(iii)& \ruletwo{\Gamma\vdash M:A\arrow B}{\Gamma\vdash
N:A}{\Gamma\vdash
MN:B} &\\[15pt]
(iv)& \ruleone{\Gamma,x:A\vdash M:B}{\Gamma\vdash\lambda x^A.M:A\arrow B} &\\[15pt]
(v)& \ruletwo{\Gamma\vdash s\tri\Delta}{\Delta\vdash
M:A}{\Gamma\vdash
 s\ci M:A} &\\[15pt]
(vi)& \Gamma\vdash id\tri\Gamma &\\[5pt]
(vii)& \Gamma,x:A\vdash\pi \tri\Gamma &\\[5pt]
(viii)& \ruletwo{\Gamma\vdash s\tri\Delta}{\Gamma\vdash
N:A}{\Gamma\vdash
\langle s\mi N\is x\rangle\tri\Delta\, ,x:A} &\\[15pt]
(ix)& \ruletwo{\Gamma\vdash s\tri\Delta}{\Delta\vdash
q\tri\Sigma}{\Gamma\vdash s\ci q\tri\Sigma} &
\end{array}
$\\[5pt]
\end{definition}
The restriction in the rule $(ii)$ is necessary  because
$\Gamma\vdash x:A$ means ``the rightmost occurrences of the
variable $x$ in the context $\Gamma$ has type $A$."

\begin{example}
$$
\ruleone{x:A\mi x:B\vdash x:B}{x:A\mi x:B\mi y:C\vdash x:B}
$$
\end{example}
\begin{example}
$$
\ruletwo{x:A\mi x:B\vdash \pi\tri x:A}{x:A\vdash x:A}{x:A\mi
x:B\vdash \pi\ci x:A}
$$
\end{example}
\begin{example}
$$
\ruletwo{\ruletwo{x:A, x:B, y:C\vdash \pi\tri x:A,x:B}{x:A,
x:B\vdash \pi\tri x:A}{x:A, x:B, y:C\vdash\pi\ci\pi\tri
x:A}}{x:A\vdash x:A}{x:A, x:B, y:C\vdash (\pi\ci\pi)\ci x:A}
$$
\end{example}
\begin{example}
$$
\ruleone{\ruleone{x:A\mi x:B\vdash x:B}{x:A\vdash\lambda
x^B.x:B\arrow B}}{\vdash\lambda x^A.\lambda x^B.x:A\arrow(B\arrow
B)}
$$
\end{example}

 There are no weakening rules except the
rule $(ii)$. But now we have an explicit weakening. For example,
we can derive $\Gamma,y:B\vdash\pi\circ M:A$ from  $\Gamma\vdash
M:A$
\begin{example}
$$\ruletwo{\Gamma,y:B\vdash\pi\tri\Gamma}{\ruledot{ \rule{5mm}{0mm} \Gamma\vdash M:A}}{\Gamma,y:B\vdash \pi\ci M:A}$$
\end{example}
If the variable $y$ does not occur in the context $\Gamma$, then
$\pi\circ M$ reduces to $M$ in some sense (more precisely,
$\pi\circ M$ and $M$ have a common reduct).

The typing rules~\ref{type} have a pleasant property: every
derivable judgement has a unique derivation. This is not true for
the usual typing rules because of weakening rules. This pleasant
property allows us to  determine uniquely the value of any
judgement in some cartesian closed category by induction over the
derivation. Assume that some objects are assigned to types. To
each context of the form
$$x_1:A_1 \mi x_2:A_2\mi\ldots\mi x_n:A_n$$
we assign the object
$$(\ldots(\mathbf{1}\times A_1)\times A_2)\times\cdots)\times
A_n)$$ where $\mathbf{1}$ is the (canonical) terminal object.\\
Denote by $A\overset{f\,\circ\, g}{\arrow} C$ the composition of
$A\overset{f}{\arrow}B$ and
$B\overset{g}{\arrow}C$.\\
 To any derivable judgement of the form
$\Gamma\vdash M:A$ we put in correspondence some arrow from
$\Gamma$ to $A$.\\
 To any derivable judgement of the form
$\Gamma\vdash s\tri\Delta$ we put in correspondence some arrow
from $\Gamma$ to $\Delta$.
\begin{definition}
 $(\Gamma\vdash
M:A)\boldsymbol{\Rightarrow}\Gamma\overset{f}{\arrow} A$
  is shorthand for  ``the arrow $\Gamma\overset{f}{\arrow} A$ corresponds to the judgement $\Gamma\vdash
  M:A$."\\
   $(\Gamma\vdash
s\tri\Delta)\boldsymbol{\Rightarrow} \Gamma\overset{f}{\arrow}
\Delta$ is shorthand for ``the arrow $\Gamma\overset{f}{\arrow}
\Delta$ corresponds to the judgement $\Gamma\vdash s\tri\Delta$."
\end{definition}
\newpage
\begin{definition}(Values of derivable judgements in cartesian closed categories).\label{modeli}\\[5pt]
 $
\begin{array}{ll}
(i)& (\Gamma,x:A\vdash x:A)\boldsymbol{\Rightarrow}  \Gamma\times A\overset{pr_2}{\arrow}A \\[5pt]
 (ii)& \ruleone{(\Gamma\vdash
x:A)\boldsymbol{\Rightarrow}  \Gamma\overset{f}{\arrow}A}{(\Gamma,y:B\vdash x:A)\boldsymbol{\Rightarrow} \Gamma\times B\overset{pr_1\,\circ\, f}{\arrow}A\using {\quad(x\not\equiv y)}}\\[20pt]
 (iii)& \ruletwo{(\Gamma\vdash M:A\arrow B)\boldsymbol{\Rightarrow}  \Gamma\overset{f}{\arrow}B^A}{(\Gamma\vdash N:A)\boldsymbol{\Rightarrow } \Gamma\overset{g}{\arrow}A}{(\Gamma\vdash
MN:B)\boldsymbol{\Rightarrow}  \Gamma\overset{\la f\mi g\ra\,\circ\, Ev}{\longrightarrow}B}\\[20pt]
 (iv)& \ruleone{(\Gamma,x:A\vdash M:B)\boldsymbol{\Rightarrow }  \Gamma\times A\overset{f}{\arrow}B}{(\Gamma\vdash\lambda x^A.M:A\arrow B)\boldsymbol{\Rightarrow}  \Gamma\overset{\Lambda (f)}{\arrow}B^A}\\[20pt]
 (v)& \ruletwo{(\Gamma\vdash s \tri\Delta)\boldsymbol{\Rightarrow}  \Gamma\overset{f}{\arrow}\Delta }{(\Delta\vdash M:A)\boldsymbol{\Rightarrow }  \Delta\overset{g}{\arrow}A}{(\Gamma\vdash
 s \ci M:A)\boldsymbol{\Rightarrow }  \Gamma\overset{f\,\circ\, g}{\arrow}A}\\[20pt]
(vi)& (\Gamma\vdash id \tri\Gamma)\boldsymbol{\Rightarrow}  \Gamma\overset{id}{\arrow}\Gamma \\[5pt]
(vii)& (\Gamma,x:A\vdash\pi \tri\Gamma)\boldsymbol{\Rightarrow}   \Gamma\times A\overset{pr_1}{\arrow}\Gamma  \\[5pt]
(viii)& \ruletwo{(\Gamma\vdash s
\tri\Delta)\boldsymbol{\Rightarrow}
 \Gamma\overset{f}{\arrow}\Delta  }{(\Gamma\vdash
N:A)\boldsymbol{\Rightarrow}
\Gamma\overset{g}{\arrow}A}{(\Gamma\vdash
\langle s\mi N\is x\rangle \tri\Delta\, ,x:A)\boldsymbol{\Rightarrow} \Gamma\overset{\la f\mi g\rangle }{\arrow}\Delta\times A}\\[20pt]
(ix)& \ruletwo{(\Gamma\vdash s \tri\Delta)\boldsymbol{\Rightarrow}
 \Gamma\overset{f}{\arrow}\Delta }{(\Delta\vdash q
\tri\Sigma)\boldsymbol{\Rightarrow}
\Delta\overset{g}{\arrow}\Sigma }{(\Gamma \vdash s \ci q
\tri\Sigma)\boldsymbol{\Rightarrow}  \Gamma\overset{f\,\circ\, g}{\arrow}\Sigma }\\
\end{array}
$
\end{definition}$ $\\[10pt]
Now we can write some  equations (untyped for simplicity).
\newpage
\begin{definition}(The calculus of equations).\label{rav}\\[5pt]
 $
\begin{array}{lll}
(Beta) & (\lambda x.M)N= \langle id\mi N\is x\rangle\ci M &\\
(Abs) & s\ci \lambda x.M=\lambda x.\langle \pi\ci s\mi x\is
x\rangle\ci M &\\
(App) & s\ci (MN)=(s\ci  M)(s\ci  N) &\\
(ConsVar) & \langle s\mi N\is x\rangle\ci x= N &\\
(New) & \langle s\mi N\is x\rangle\ci y=  s\ci y & (x\not\equiv y)\\
(IdVar) &  id\ci x= x &\\
(Clos) & s\ci q\ci M= (s\ci q)\ci M &\\
(Ass) & s\ci q\ci  r= (s\ci q)\ci r &\\
(IdR) & s\ci id= s &\\
(IdShift) & id\ci \pi=\pi &\\
(ConsShift) & \langle s\mi N\is x\rangle\ci \pi= s &\\
(Map) & s\ci \langle q\mi N\is x\rangle=\langle
s\ci q \mi s\ci N\is x\rangle &\\
(\alpha) & \lambda x.M=\lambda
y.\langle\pi\mi y\is x\rangle\ci M & (x,y \text{ are arbitrary})\\
[5pt]
\end{array}$
\end{definition}

Here $s,q,r$ are substitutions,\\
$\begin{array}{lll}
 s\ci q\ci M & \text{is shorthand for} &
s\ci(q\ci M)\\
 s\ci q\ci r & \text{is shorthand for} & s\ci(q\ci r)\\
 \la
s\ci q\mi s\ci N\is x\ra & \text{is shorthand for} & \la (s\ci
q)\mi(s\ci N)\is x\ra
\end{array}$\\

 The names of the equations are taken from \cite{Abadi},
but partially reversed ($ConsVar$ instead of $VarCons$ and so on)
because of the reversed notation. The equations $New$  and
$\alpha$ are new\footnote{After the article was written, Johan G. 
Granstr\"{o}m pointed me to his PhD thesis~\cite{Granstrom} with a
very similar calculus in Chapter V.}.

When a substitution is applied to a variable, the rightmost
occurrence of this variable works. See the following example
$$\la id\mi M\is x\mi N\is x\mi L\is y\ra\ci x=_{New}\la id\mi M\is
x\mi N\is x\ra\ci x=_{ConsVar}N$$
 I want to stress that there is
no restriction on the variables in $(\alpha)$. For example, we can
write
 $$\lambda x.M=\lambda
x.\la\pi\mi x\is x\ra\ci M$$

The following special case of $(\alpha)$ is important
$$
\begin{array}{ll}
\lambda x.y=\lambda x.\la\pi\mi x\is x\ra\ci y & \quad(x\not\equiv y)\\
[5pt]
\end{array}$$

Applying $New$ to the right part, we obtain\\ $
\begin{array}{lccc}
\quad(\pi)& \qquad & \lambda x.y=\lambda x.\pi\ci y & \quad(x\not\equiv y)\\
[5pt]
\end{array}$

Now we can compute lambda-terms.
\begin{example}\label{primer}
$$\begin{array}{ll}
(\lambda x y.x)\,y & \\
= \langle id\mi y\is x\rangle\ci \lambda y.x &
Beta\\
=\lambda y.\langle\pi\ci \langle id\mi y\is x\rangle\mi y\is
y\rangle\ci x & Abs\\
=\lambda y.(\pi\ci \langle id\mi y\is x\rangle)\ci x &
New\\
=\lambda
y.\langle\pi\ci id\mi \pi\ci y\is x\rangle\ci x & Map\\
= \lambda y.\pi\ci y & ConsVar\\
=\lambda z.\langle\pi\mi z\is y\rangle\ci \pi\ci y &
\alpha\\
=\lambda z.(\langle\pi\mi z\is y\rangle\ci \pi)\ci y & Clos\\
=\lambda z.\pi\ci y & ConsShift\\
=\lambda z.y & \pi
\end{array}
$$
\end{example}

We show by examples how to define free variables of terms. Our
calculus has an unexpected feature: the variable $x$ can occur
freely in a term of the form $\lambda x.M$. To each free
occurrence of $x$ in $M$ assign its \emph{level} (it is not the De
Brujn level), which is a natural number $\geqslant 1$. The only
occurrence of $x$ in the term $x$ has  level $1$. We can
immediately bind this occurrence and get $\lambda x.x$. The only
occurrence of $x$ in the term $\pi\ci x$ has  level $2$. The
rightmost occurrence of $x$ in the term $\lambda x.\pi\ci x$ is
free and has  level $1$. The rightmost occurrence of $x$ in the
term $\lambda x.\lambda x.\pi\ci x$ is bound. The only occurrence
of $x$ in the term $\pi\circ\pi\ci x$ has  level $3$. The
rightmost occurrence of $x$ in the term $\lambda x.\pi\circ\pi\ci
x$ is free and has level $2$. The rightmost occurrence of $x$ in
the term $\lambda x.\lambda x.\pi\circ\pi\ci x$ is free and has
level $1$. The rightmost occurrence of $x$ in the term $\lambda
x.\lambda x.\lambda x.\pi\circ\pi\ci x$ is bound.

Because  the rightmost occurrence of $y$ in the term $\lambda
y.\pi\circ y$ is free, this term is $\alpha$-equal to the term
$\lambda z.\pi\circ y$. This renaming of  the bound variable is
done in Example~~\ref{primer}.
$$\begin{array}{ll}
\lambda y.\pi\ci y & \\
=\lambda z.\langle\pi\mi z\is y\rangle\ci \pi\ci y &
\alpha\\
=\lambda z.(\langle\pi\mi z\is y\rangle\ci \pi)\ci y & Clos\\
=\lambda z.\pi\ci y & ConsShift
\end{array}
$$

I do not have good rewrite rules for the calculus~\ref{rav}, hence
I propose a different approach. We change the language. Now each
symbol $\pi$  is equipped with a variable as a subscript
$(\pi_x,\pi_y,\pi_z\ldots)$.  The sets of untyped terms and
substitutions are defined inductively as follows:
\begin{align*}
M,N::&= x \mid  MN \mid \lambda x. M \mid s\ci M \\
s,q::&= id \mid \pi_{x} \mid \langle s\mi  N\is x \rangle \mid
s\ci q
\end{align*}
where the symbol $x$ denotes an arbitrary variable.

The sets of typed terms and substitutions are defined inductively
as follows:
\begin{align*}
M,N::&= x \mid  MN \mid \lambda x^A .M \mid s\ci M \\
s,q::&= id \mid \pi_x \mid \langle s\mi  N\is x \rangle \mid s\ci
q
\end{align*}
Rule $(vii)$ of Definition~\ref{type} is changed to\\
$
\begin{array}{lll}
(vii)& \Gamma,x:A\vdash\pi_x \tri\Gamma &\\
\end{array}
$\\[5pt]
Rule $(vii)$ of Definition~\ref{modeli} is changed to\\
$
\begin{array}{ll}
(vii)& (\Gamma,x:A\vdash\pi_x \tri\Gamma)\boldsymbol{\Rightarrow}   \Gamma\times A\overset{pr_1}{\arrow}\Gamma  \\[5pt]
\end{array}
$

\begin{example}
$$
\ruletwo{x:A,x:B\vdash \pi_x\tri x:A}{x:A\vdash x:A}{x:A,x:B\vdash
\pi_x\ci x:A}
$$
\end{example}
\begin{example}
$$
\ruletwo{\ruletwo{x:A,x:B,y:C\vdash \pi_y\tri
x:A,x:B}{x:A,x:B\vdash \pi_x\tri
x:A}{x:A,x:B,y:C\vdash\pi_y\ci\pi_x\tri x:A}}{x:A\vdash
x:A}{x:A,x:B,y:C\vdash (\pi_y\ci\pi_x)\ci x:A}
$$
\end{example}

The calculus~\ref{rav} is a draft. We write a similar calculus in
the new language, this is $\lambda\pi$.
\begin{definition}(The calculus $\lambda\pi$ without several rules).\\
 $
\begin{array}{lll}
(Beta) & (\lambda x.M)N\rightarrow \langle id\mi N\is x\rangle\ci M &\\
(Abs) & s\ci \lambda x.M\rightarrow\lambda x.\langle \pi_{x}\ci
s\mi x\is
x\rangle\ci M &\\
(App) & s\ci (MN)\rightarrow(s\ci  M)(s\ci  N) &\\
(ConsVar) & \langle s\mi N\is x\rangle\ci x\rightarrow N &\\
(New) & \langle s\mi N\is x\rangle\ci y\rightarrow  s\ci y & (x\not\equiv y)\\
(IdVar) &  id\ci x\rightarrow x &\\
(Clos) & s\ci q\ci M\rightarrow (s\ci q)\ci M &\\
(Ass) & s\ci q\ci  r\rightarrow (s\ci q)\ci r &\\
(IdR) & s\ci id\rightarrow s &\\
(IdShift) & id\ci \pi_{x}\rightarrow\pi_{x} &\\
(ConsShift) & \langle s\mi N\is x\rangle\ci \pi_{x}\rightarrow s &\\
(Map) & s\ci \langle q\mi N\is x\rangle\rightarrow\langle
s\ci q \mi s\ci N\is x\rangle &\\
(\pi_1) & \pi_{x}\ci y\rightarrow y & (x\not\equiv y)\\
(\pi_2) & (s\ci\pi_{x})\ci y\rightarrow s\ci y & (x \not\equiv
y)\\
(\alpha_1) & \lambda x.M\rightarrow\lambda y.\langle\pi_{y}\mi
y\is x\rangle\ci M & (*)
\end{array}$
\\
where $(*)$ is some restriction on the variables: if the variable
$x$ occurs freely in $\lambda x.M$, we can rename $x$ to a ``good"
variable.
\end{definition}
Example~\ref{primer} now looks like this:
\begin{example}\label{primer2}
$$\begin{array}{ll}
(\lambda x y.x)\,y & \\
\rightarrow\langle id\mi y\is x\rangle\ci \lambda y.x &
Beta\\
\rightarrow\lambda y.\langle\pi_y\ci \langle id\mi y\is
x\rangle\mi y\is y\rangle\ci x & Abs\\
 \rightarrow\lambda
y.(\pi_y\ci \langle id\mi y\is x\rangle)\ci x &
New\\
\rightarrow\lambda y.\langle\pi_y\ci id\mi \pi_y\ci y\is
x\rangle\ci x & Map\\
 \rightarrow \lambda y.\pi_y\ci
y & ConsVar\\
\rightarrow\lambda z.\langle\pi_z\mi z\is y\rangle\ci \pi_y\ci y &
\alpha_1\\
\rightarrow\lambda z.(\langle\pi_z\mi z\is y\rangle\ci \pi_y)\ci y
& Clos\\
 \rightarrow\lambda z.\pi_z\ci
y & ConsShift\\
\rightarrow\lambda z.y & \pi_1
\end{array}
$$
\end{example}

We were able to apply $\alpha_1$ because $y$ occurs freely in
$\lambda y.\pi_y\circ y$ (the rightmost occurrence is free).

Let's try to write the formal analogue of Definition~\ref{type}
for untyped terms and substitutions. Contexts are now simply
finite lists of variables with multiplicity (i.e., repetitions are
permitted).

A \emph{judgement} is now an expression of the form $\Gamma\vdash
M$ or of the form $\Gamma\vdash s\tri\Delta$, where $\Gamma$ and
$\Delta$
are contexts,  $M$ is a term, and $s$ is a substitution.\\
$\Gamma\vdash M$ means that $M$ is a well-formed term in the
context
$\Gamma$.\\
$\Gamma\vdash s\tri\Delta$ means that $s$ is a well-formed
substitution for $\Delta$ over $\Gamma$.
\begin{definition}(Well-formed terms and substitutions).\label{well}\\
 $
\begin{array}{lll}
(i)& \Gamma,x\vdash x &\\[5pt]
(ii)& \ruleone{\Gamma\vdash
x}{\Gamma,y\vdash x} & (x\not\equiv y) \\[15pt]
(iii)& \ruletwo{\Gamma\vdash M}{\Gamma\vdash N}{\Gamma\vdash
MN} &\\[15pt]
(iv)& \ruleone{\Gamma,x\vdash M}{\Gamma\vdash\lambda x.M} &\\[15pt]
(v)& \ruletwo{\Gamma\vdash s\tri\Delta}{\Delta\vdash
M}{\Gamma\vdash
 s\ci M} &\\[15pt]
(vi)& \Gamma\vdash id\tri\Gamma &\\[5pt]
(vii)& \Gamma,x\vdash\pi_{x} \tri\Gamma &\\[5pt]
(viii)& \ruletwo{\Gamma\vdash s\tri\Delta}{\Gamma\vdash
N}{\Gamma\vdash
\langle s\mi N\is x\rangle\tri\Delta ,x} &\\[15pt]
(ix)& \ruletwo{\Gamma\vdash s\tri\Delta}{\Delta\vdash
q\tri\Sigma}{\Gamma\vdash s\ci q\tri\Sigma} &
\end{array}
$
 \end{definition}
\begin{example}
$$
\ruleone{x,x\vdash x}{x,x,y\vdash x}
$$
\end{example}
\begin{example}
$$
\ruletwo{x,x\vdash \pi_x\tri x}{x\vdash x}{x,x\vdash \pi_x\ci x}
$$
\end{example}
\begin{example}
$$
\ruletwo{\ruletwo{x,x,y\vdash \pi_y\tri x,x}{x,x\vdash \pi_x\tri
x}{x,x,y\vdash\pi_y\ci\pi_x\tri x}}{x\vdash x}{x,x,y\vdash
(\pi_y\ci\pi_x)\ci x}
$$
\end{example}
\begin{example}
$$
\ruleone{\ruleone{x,x\vdash x}{x\vdash\lambda x.x}}{\vdash\lambda
x.\lambda x.x}
$$
\end{example}

 All usual $\lambda$-terms (without explicit substitutions) are well-formed. But there are some restrictions  on subscripts of the symbols $\pi_x$. For example, a term of the form $\lambda x.\pi_x\circ M$ is well-formed if
 $M$is well-formed
$$
\ruleone{\ruletwo{\Gamma,x\vdash\pi_x\tri\Gamma}{\ruledot{\Gamma\vdash
M}}{\Gamma,x\vdash \pi_x\ci M}}{\Gamma\vdash\lambda x.\pi_x\ci M}
$$\\[5pt]
but a term of the form $\lambda x.\pi_y\circ M$ is never
well-formed
$$
\ruleone{\ruletwo{\Gamma,y\vdash
\pi_y\tri\Gamma}{\ruledot{\Gamma\vdash M}}{\Gamma,x\vdash \pi_y\ci
M \using(?)}}{\Gamma\vdash\lambda x.\pi_y\ci M}
$$
Reducts of well-formed terms and substitutions are well-formed,
hence reducts of usual $\lambda$-terms are well-formed. We will
work only with well-formed terms and substitutions.

But there is a problem: we can  not reduce such term as
$\pi_y\circ y$. We can reduce $\lambda y.\pi_y\circ y$ (to
$\lambda z.y$), but not $\pi_y\circ y$. It is unpleasant to have
such normal forms. Hence we introduce a new idea. So far we have
one step reductions $M_1\arrow M_2$ and $s_1\arrow s_2$ defined on
the sets of terms and substitutions respectively. We introduce
also a one step reduction\\
 $\Gamma_1\vdash M_1\ri\Gamma_2\vdash
M_2$ defined on the set of judgements of the form $\Gamma\vdash
M$. Really we need only derivable  judgements in the sense of
Definition~\ref{well}.

\begin{definition}(Compatible closure).\\[10pt]
$
\begin{array}{ccc}
& \ruleone{ M_1\arrow M_2}{ \lambda
x.M_1\arrow\lambda x.M_2}\\[20pt]
& \ruleone{ M_1\arrow M_2}{ M_1N\arrow  M_2N} & \ruleone{
N_1\arrow N_2}{
MN_1\arrow  MN_2}\\[20pt]
& \ruleone{  s_1\arrow s_2}{s_1\ci M\arrow s_2\ci M} & \ruleone{
M_1\arrow
M_2}{s\ci M_1\arrow s\ci M_2}\\[20pt]
& \ruleone{ s_1\arrow s_2}{ \langle s_1\mi N\is x\rangle\arrow
\langle s_2\mi N\is x\rangle} & \ruleone{ N_1\arrow N_2}{ \langle
s\mi N_1\is x\rangle\arrow \langle
s\mi N_2\is x\rangle}\\[20pt]
& \ruleone{ s_1\arrow s_2}{ s_1\ci q\arrow s_2\ci q} & \ruleone{
q_1\arrow q_2}{ s\ci q_1\arrow s\ci
q_2}\\[20pt]
& \ruleone{M_1\arrow M_2}{\Gamma\vdash M_1\ri
\Gamma\vdash M_2}\\[20pt]
\end{array}
$
\end{definition}
At last, we add one more rewrite rule (called $\alpha_2$), which
can be applied to a judgement of the form $\Gamma\vdash M$ and
renames a variable in the context $\Gamma$. For example, the term
$\pi_y\circ y$ can be well-formed only in a context of the form
$\Delta,y$. We can apply $\alpha_2$ to the judgement
$\Delta,y\vdash \pi_y\circ y$ and obtain the judgement
$\Delta,z\vdash\pi_z\circ y$, which then reduces to
$\Delta,z\vdash y$. We denote by $\Lambda\pi$ the set of derivable
 judgements of the form $\Gamma\vdash M$. For $\lambda\pi$ this set  is like
$\Lambda$ for $\lambda\beta$ and $\ri$ is the main one step
reduction.

The rest of the paper is organized as follows. Section~2 defines
the sets of contexts, terms, and substitutions. Section~3 provides
a definition of free variables. Section~4 introduces the calculus
$\lambda\pi$.  Section~5  proves Subject reduction.  Section~6
proves several useful properties of $\lambda\pi$. Section~7
compares $\lambda\pi$ with $\lambda\sigma$ from~\cite{Abadi}.
Section~8 defines the $\alpha$-equivalence. Section~9 proves that
$\lambda\pi$ is confluent. Section~10 shows that any computation
without $Beta$ is strongly normalized.
\newpage

%% file: section1.tex
\section{Terms and substitutions}
For accuracy, we will use metavariables for variables. For
example, beta-reduction rule would be written as:\\
$(\lambda \x.M)N\arrow\la id\mi N\is \x\ra\ci M$,\\
where $\x$ is a metavariable for variables, $M$ and $N$ are
metavariables for terms. Replacing $\x$ by the variable $x$, $M$
by the term $xx$, and $N$ by the term $y$, we obtain the following
concrete example of beta-reduction:\\
$(\lambda x.xx)y\arrow\la id\mi y\is x\ra\ci (xx)$

For simplicity  we  will work with  the untyped calculus. However,
we will use contexts.
\begin{definition}\label{pervoe} The symbols $x,y,z,\ldots$ are \emph{variables}. The symbols $M,N,L$ range over \emph{terms},
$s,q,r$  range over \emph{substitutions}, and $\x,\y,\z$  range
over variables (they are \emph{metavariables}). The sets of terms
and substitutions are defined inductively as follows:
\begin{align*}
M,N::&= \x \mid  MN \mid \lambda \x. M \mid s\ci M\\
s,q::&= id \mid \pi_{\x} \mid \langle s\mi  N\is \x \rangle \mid
s\ci q
\end{align*}
Note that\\[5pt]
 $\begin{array}{llll}
 s\ci M & \text{corresponds to} & M[s] &\text{ from ~\cite{Abadi}};\\
 s\ci q & \text{corresponds to} & q\ci s &\text{ from ~\cite{Abadi}};\\
\la s\mi N\is \x\ra & \text{corresponds to} & N\cdot s &\text{ from ~\cite{Abadi}};\\
\pi_{\x} & \text{corresponds to} & \uparrow &\text{ from
~\cite{Abadi}}.
   \end{array}$\\
   \end{definition}
   \begin{convention}\label{notation2}
 Outermost parentheses are not written.\\
  Outermost parentheses around $s$ in  $\la s\mi N\is
\x\ra$ are not written.\\
 Outermost parentheses around $N$ in  $\la s\mi N\is
\x\ra$ are not written.
 \end{convention}
\begin{convention}\label{notation1}$ $\\[3pt]
$\begin{array}{lll} MN_1\ldots  N_k & \text{is shorthand for} &
 ((MN_1)\ldots)N_k\\
 \lambda \x_1\ldots\x_k.M & \text{is shorthand for}
 & \lambda \x_1.(\ldots(\lambda \x_k.M))\\
 \lambda \x.MN_1\ldots N_k & \text{is shorthand for}
 & \lambda \x.(MN_1\ldots N_k)\\
s\ci MN_1\ldots N_k & \text{is shorthand for}
 & s\ci(MN_1\ldots N_k)\\
 \lambda \x.s\ci M & \text{is shorthand for}
 & \lambda \x.(s\ci M)\\
 s\ci \lambda \x.M & \text{is shorthand for}
 & s\ci(\lambda \x.M)\\
 s_1\ci\ldots\ci s_k\ci s_{k+1} & \text{is shorthand for}
 & s_1\ci(\ldots\ci(s_k\ci s_{k+1}))\\
 s_1\ci \ldots\ci s_k\ci M & \text{is shorthand for}
 & s_1\ci(\ldots\ci(s_k\ci M))\\
 \langle s\mi N_1\is \y_1\mi \ldots\mi N_n\is \y_n \rangle &
\text{is shorthand for} & \langle\langle\ldots\la\langle s\mi
N_1\is \y_1\rangle\mi
 N_2\is \y_2\rangle\mi\ldots\rangle\mi N_n\is \y_n\rangle
 \end{array}$
\end{convention}
\begin{example}$\begin{array}{lll}
id\ci id\ci x & \text{is shorthand for} & id\ci (id\ci x)
\end{array}$
\end{example}
\begin{example}$\begin{array}{lll}
\lambda x.id\ci y & \text{is shorthand for} & \lambda x.(id\ci y)
\end{array}$
\end{example}
\begin{example}$\begin{array}{lll}
id\ci x(yz) & \text{is shorthand for} & id\ci (x(yz))
\end{array}$
\end{example}
\begin{example}$ $\\
$\begin{array}{lll} \la id\ci id\mi id\ci y\is x\ra & \text{is
shorthand for} & \la (id\ci id)\mi(id\ci y)\is x\ra\
\end{array}$
\end{example}
\begin{example}$ $\\
$\begin{array}{lll} id\ci\lambda x.\pi_x\ci\lambda y.z & \text{is
shorthand for} & id\ci(\lambda x.(\pi_x\ci(\lambda y.z)))
\end{array}$
\end{example}
\begin{example}$ $\\
$\begin{array}{lll} \la id\mi y\is x\mi z\is x\ra & \text{is
shorthand for} & \la\la id\mi y\is x\ra\mi z\is x\ra
\end{array}$
\end{example}

For  a more precise definition of terms and substitutions see
Section 11.
 \begin{definition}
A \emph{context} is a possibly empty, finite list of variables
with multiplicity (i.e., repetitions are permitted). The symbols
$\Gamma,\Delta,\Sigma,\Psi$ range over contexts.
\end{definition}
\begin{example}
The list $x,x,y$ is a context.
\end{example}
\begin{definition}
A \emph{judgement} is an expression of the form $\Gamma\vdash M$
or of the form $\Gamma\vdash s\tri\Delta$.
\end{definition}
A judgement of the form $\Gamma\vdash M$  means ``$M$ is a
well-formed term in the context $\Gamma$." A judgement of the form
$\Gamma\vdash s\tri\Delta$  means ``$s$ is a well-formed
substitution for $\Delta$ over $\Gamma$."
\begin{definition}\label{defsequentsder}(The inference rules for judgements).\\
$
\begin{array}{lll}
(i)& \Gamma,\x\vdash \x &\\[5pt]
(ii)& \ruleone{\Gamma\vdash
\x}{\Gamma,\y\vdash \x} & (\x\neq \y) \\[15pt]
(iii)& \ruletwo{\Gamma\vdash M}{\Gamma\vdash N}{\Gamma\vdash
MN} &\\[15pt]
(iv)& \ruleone{\Gamma,\x\vdash M}{\Gamma\vdash\lambda \x.M} &\\[15pt]
(v)& \ruletwo{\Gamma\vdash s\tri\Delta}{\Delta\vdash
M}{\Gamma\vdash
 s\ci M} &\\[15pt]
(vi)& \Gamma\vdash id\tri\Gamma &\\[5pt]
(vii)& \Gamma,\x\vdash\pi_{\x} \tri\Gamma &\\[5pt]
(viii)& \ruletwo{\Gamma\vdash s\tri\Delta}{\Gamma\vdash
N}{\Gamma\vdash
\langle s\mi N\is \x\rangle\tri\Delta ,\x} &\\[15pt]
(ix)& \ruletwo{\Gamma\vdash s\tri\Delta}{\Delta\vdash
q\tri\Sigma}{\Gamma\vdash s\ci q\tri\Sigma} &
\end{array}
$\\
\end{definition}
Here $\x\neq\y$ means that $\x$ and $\y$ denote distinct
variables.
\begin{example}
$$
\ruletwo{x,x,y\vdash \pi_y\tri x,x}{\ruletwo{x,x\vdash \pi_x\tri
x}{x\vdash x}{x,x\vdash \pi_x\ci x}}{x,x,y\vdash \pi_y\ci\pi_x\ci
x}
$$
\end{example}
\begin{example}
$$
\ruletwo{\ruleone{\ruledot{\Gamma,x\vdash M}}{\Gamma\vdash\lambda
x.M}}{\ruledot{\Gamma\vdash N}}{\Gamma\vdash (\lambda x.M)N}
$$
\end{example}
\begin{example}
$$
\ruletwo{\ruletwo{\Gamma\vdash id\tri\Gamma}{\ruledot{\Gamma\vdash
N}}{\Gamma\vdash\la id\mi N\is
x\ra\tri\Gamma,x}}{\ruledot{\Gamma,x\vdash M}}{\Gamma\vdash\la
id\mi N\is x\ra\ci M}
$$
\end{example}
\begin{lemma}[Generation lemma]$ $\\
Each derivation of $\,\Gamma,\x\vdash \x$ is an application of the
rule~(i).\\
Each derivation of $\,\Gamma,\y\vdash \x$ $($where $\x\neq \y)$ is
an application of the rule~(ii) to some derivation of
$\,\Gamma\vdash
\x$.\\
Each derivation of $\,\Gamma\vdash MN$ is an application of the
rule~(iii) to some  derivations of $\,\Gamma\vdash M$ and
$\,\Gamma\vdash N$.\\
Each derivation of $\,\Gamma\vdash \lambda \x.M$  is an
application of the rule~(iv) to some derivation of
$\,\Gamma,\x\vdash M$.\\
Each derivation of $\,\Gamma\vdash s\ci M$ is an application of
the rule~(v) to some derivations of $\,\Gamma\vdash s\tri\Delta$
and $\,\Delta\vdash M$ for some
$\Delta$.\\
Each derivation of $\,\Gamma\vdash id\tri\Delta$ is an application
of the
rule~(vi), where $\Delta$ coincides with $\Gamma$.\\
Each derivation of $\,\Delta\vdash \pi_{\x}\tri\Gamma$ is an
application of the rule~(vii), where $\Delta$ coincides with $\Gamma,\x$.\\
Each derivation of $\,\Gamma\vdash\langle s\mi N\is
\x\rangle\tri\Sigma$ is an application of the rule~(viii) to some
derivations of
$\,\Gamma\vdash s\tri\Delta$ and $\,\Gamma\vdash N$ for some $\Delta$, where $\Sigma$ coincides with $\Delta,\x$.\\
Each derivation of $\,\Gamma\vdash s\ci q\tri\Sigma$ is an
application of the rule~(ix) to some derivations of
$\,\Gamma\vdash s\tri\Delta$ and $\,\Delta\vdash q\tri\Sigma$ for
some $\Delta$.
\end{lemma}
\begin{proof}
The proof is straightforward.
\end{proof}
\begin{lemma}
 If a judgement of the form $\Gamma\vdash s\tri\Delta$ is
derivable, then $\Delta$ is uniquely defined for given $\Gamma$
and $s$.
\end{lemma}
\begin{proof}The proof is by induction over the structure of
$s$ (see Definition~\ref{pervoe}).\\
Case 1: $s$ is $id$. This implies that $\Delta$ coincides with
$\Gamma$.\\
Case 2: $s$ has the form $\pi_{\x}$ for some $\x$. This implies
that $\Gamma$
coincides with $\Delta,\x$.\\
Case 3: $s$ has the form $\langle q\mi N\is \x\rangle$ for some
$q,N,\x$. By Generation lemma,  we can derive $\Gamma\vdash
q\tri\Sigma$ for some $\Sigma$. By the induction hypothesis,
$\Sigma$ is uniquely defined for $\Gamma$ and $q$. Then $\Delta$
coincides with
$\Sigma,\x$.\\
Case 4: $s$ has the form $q\ci r$ for some $q,r$. By Generation
lemma, we can derive $\Gamma\vdash q\tri\Sigma$ for some $\Sigma$.
By the induction hypothesis, $\Sigma$ is uniquely defined for
$\Gamma$ and $q$. By Generation lemma, we can derive $\Sigma\vdash
r\tri\Delta$, where $\Delta$ is uniquely defined for $\Sigma$ and
$r$.
\end{proof}
\begin{lemma}For any derivable judgement, there is a unique
derivation.
\end{lemma}
\begin{lemma}The problem of derivability for judgements is
decidable.
\end{lemma}
\begin{proof} We try to construct a derivation from the bottom up.
\end{proof}
\begin{example}\label{notwellformterm}Not each term is well-formed in any context. A
term of the form $\lambda \x.\pi_{\y}\ci M $ is not well-formed in
any context if $\x\neq \y$.
$$
\ruleone{\ruletwo{\Gamma,y\vdash
\pi_y\tri\Gamma}{\ruledot{\Gamma\vdash M}}{\Gamma,x\vdash \pi_y\ci
M \using(?)}}{\Gamma\vdash\lambda x.\pi_y\ci M}
$$
\end{example}
\begin{example}A term of the form $(\pi_{\x}\ci M)(\pi_{\y}\ci N)$ is not
well-formed in any context if $\x\neq \y$.
$$
\ruletwo{\ruletwo{\Gamma,x\vdash\pi_x\tri\Gamma}{\ruledot{\Gamma\vdash
M}}{\Gamma,x\vdash\pi_x\circ M}}
{\ruletwo{\Gamma,y\vdash\pi_y\tri\Gamma}{\ruledot{\Gamma\vdash
N}}{\Gamma,y\vdash\pi_y\circ N}} {?\vdash(\pi_x\circ M)(\pi_y\circ
N)}
$$

\end{example}
\begin{example}\label{bedsubst}A substitution of the form $\la s\mi N\is
\x\ra\ci\pi_{\y}$ is not well-formed in any contexts if $\x\neq
\y$.
$$
\ruletwo{\ruletwo{\ruledot{\Gamma\vdash s\tri\Delta}}{\Gamma\vdash
N}{\Gamma\vdash\la s\mi N\is
x\ra\tri\Delta,x}}{\Delta,y\vdash\pi_y\tri\Delta\using(?)}{\Gamma\vdash\la
s\mi N\is x\ra\ci\pi_y\tri\Delta}
$$
\end{example}
\newpage

%% file: freevariables.tex
\section{Free variables}
Consider some term $M$ and some variable $\x$. To each free
occurrence of $\x$ in $M$ assign its \emph{level}, which is a
natural number $\geqslant 1$. The only occurrence of $x$ in the
term $x$ has  level $1$. We can immediately bind this occurrence
and get $\lambda x.x$. The only occurrence of $x$ in the term
$\pi_y\ci x$ has  level $2$. We can write the term $\lambda
x.\pi_y\ci x$, but this term is not well-formed (see
Example~\ref{notwellformterm}). If we want to bind this occurrence
and get a well-formed term, we must write $\lambda x y.\pi_y\ci
x$, hence the level is $2$. The only occurrence of $x$ in the term
$\pi_z\ci\pi_y\ci x$ has  level $3$. The simplest way to bind this
occurrence and get a well-formed term is $\lambda x y
z.\pi_z\ci\pi_y\ci x$. Subscripts of the symbols $\pi_a$ are not
considered as free occurrences.
\begin{definition} The symbols $\mathcal{A,B}$ range over infinite
sequences of sets
$$\langle \mathcal{A}_1,\mathcal{A}_2,\mathcal{A}_3,\ldots\rangle\in Sets^\omega$$
$$\langle \mathcal{B}_1,\mathcal{B}_2,\mathcal{B}_3,\ldots\rangle\in Sets^\omega$$
By $\mathcal{A}\cup\mathcal{B}$ denote
$$\langle\mathcal{A}_1\cup\mathcal{B}_1,\mathcal{A}_2\cup\mathcal{B}_2,\mathcal{A}_3\cup\mathcal{B}_3,\ldots\rangle$$
\end{definition}
In fact, we need only finite sets whose elements are variables. To
each term $L$ we assign an infinite sequence of sets
$$FV(L)\equiv\langle
 FV_1(L),FV_2(L),FV_3(L),\ldots\rangle\in Sets^\omega$$
 The variables from the set $FV_i(L)$ have free occurrences of  level
 $i$ in $L$.
 The set of free variables of $L$ is $\bigcup_{i\geqslant 1}
 FV_i(L)$.
\begin{definition}(Free variables of terms). By definition, put\label{FV}\\
$
\begin{array}{ll}
(i) & FV(\x)=\langle \{\x\},\emptyset,\emptyset,\ldots\rangle\\
(ii) & FV(MN)=FV(M)\cup FV(N)\\
(iii) & FV(\lambda \x.M)=O_{\lambda \x}(FV(M))\\
(iv) & FV(s\ci M)=O_s(FV(M)),
\end{array}
$\\[5pt]
where\\[5pt] $
\begin{array}{ll}
(v) & O_{\lambda \x}(\mathcal{A})=
 \langle(\mathcal{A}_1\setminus\{\x\})\cup\mathcal{A}_2\,,\mathcal{A}_3,\mathcal{A}_4,\ldots\rangle\\
(vi) & O_{id}(\mathcal{A})=\mathcal{A}\\
(vii) &
O_{\pi_{\x}}(\mathcal{A})=\langle\emptyset,\mathcal{A}_1,\mathcal{A}_2,\ldots\rangle\\
(viii) & O_{s\circ q}(\mathcal{A})=O_s(O_{q}(\mathcal{A}))\\
(ix) & O_{\langle s\mi N\is \x\rangle}(\mathcal{A})=O_s(O_{\lambda
\x}(\mathcal{A}))\cup
 FV(N)
 \end{array}$
\end{definition}
\begin{corollary}$ $\\
 $
\begin{array}{l}
 FV_1(\lambda \x.M)=(FV_1(M)\setminus\{\x\})\cup FV_2(M)\\
 FV_{n+1}(\lambda \x.M)=FV_{n+2}(M)\quad (n\geqslant 1)\\
 FV(id\ci M)= FV(M)\\
 FV_1(\pi_{\x}\ci M)=\emptyset \\
 FV_{n+1}(\pi_{\x}\ci M)= FV_{n}(M)\quad (n\geqslant 1)\\
 FV(\langle s\mi N\is \x\rangle\ci M)= O_s(O_{\lambda
\x}(FV(M)))\cup FV(N)
\end{array}
$
\end{corollary}
\begin{example}
$$FV(x)=\langle\{x\},\emptyset,\emptyset,\ldots\rangle$$
\end{example}
\begin{example}
$$FV(\pi_y\ci x)=\langle\emptyset,\{x\},\emptyset,\ldots\rangle$$
\end{example}
\begin{example}
$$FV(\pi_z\ci\pi_y\ci x)=\langle\emptyset,\emptyset,\{x\},\emptyset,\ldots\rangle$$
\end{example}
\begin{example}
$$FV(\lambda z.\pi_z\ci\pi_y\ci x)=\langle\emptyset,\{x\},\emptyset,\ldots\rangle$$
\end{example}
\begin{example}
$$FV(\lambda y z.\pi_z\ci\pi_y\ci x)=\langle\{x\},\emptyset,\emptyset,\ldots\rangle$$
\end{example}
\begin{example}
$$FV(\lambda x y z.\pi_z\ci\pi_y\ci x)=\langle\emptyset,\emptyset,\emptyset,\ldots\rangle$$
\end{example}
\begin{example}
$$FV(x\,(\pi_z\ci\pi_y\ci x))=\langle\{x\},\emptyset,\{x\},\emptyset,\ldots\rangle$$
\end{example}
\begin{example}
$$FV(\lambda
z.x\,(\pi_z\ci \pi_y\ci
x))=\langle\{x\},\{x\},\emptyset,\ldots\rangle$$
\end{example}
\begin{example}
$$FV(\lambda y z.x\,(\pi_z\ci\pi_y\ci x))=\langle\{x\},\emptyset,\emptyset,\ldots\rangle$$
\end{example}
\begin{example}
$$FV(\lambda x y
z.x\,(\pi_z\ci\pi_y\ci
x))=\langle\emptyset,\emptyset,\emptyset,\ldots\rangle$$
\end{example}
Warning! May be that $\x\in\bigcup_{i\geqslant 1}
 FV_i(\lambda \x.M)$.
  \begin{example}
$FV(\pi_x\ci x)=\la\emptyset,\{x\},\emptyset,\emptyset,\ldots\ra$
\end{example}
\begin{example}
$FV(\lambda x.\pi_x\ci x)=\la\{x\},\emptyset,\emptyset,\ldots\ra$
\end{example}
In fact, the term $\lambda x.\pi_x\ci x$ is $\alpha$-equal to
$\lambda y.\pi_y\ci x$.
\begin{lemma}\label{corollary23} $(s\ci q)\ci M$ and $s\ci q\ci M$ have
the same
 $FV$ .
\end{lemma}
\begin{proof}The proof is straightforward.
\end{proof}
\begin{lemma}\label{corollary3} $\langle s\mi N\is \x\rangle\circ M$ and $(s\circ \lambda
\x.M)N$ have the same $FV$.
\end{lemma}
\begin{proof}The proof is straightforward.
\end{proof}
\begin{convention}Since $O_{\pi_{\x}}$ and $O_{\pi_{\y}}$ are the
same for any $\x,\y$, we will simply write $O_{\pi}$.
\end{convention}
\begin{definition}\label{mon}
 $\mathcal{A}\subseteq\mathcal{B}$ is shorthand for
``$\mathcal{A}_i\subseteq\mathcal{B}_i$ for all $i\geqslant 1$."
\end{definition}
\begin{lemma}\label{monotone}$O_{\lambda \x}$ and $O_{\pi}$ are monotone
operators with respect to $\subseteq$ (for any $\x$).
\end{lemma}
\begin{proof}The proof is straightforward.
\end{proof}
\begin{corollary}$O_s$ is monotone with respect to $\subseteq$ for any $s$.
\end{corollary}
\begin{definition}We define  $\lambda\Gamma.M$  as follows:
\\
$
\begin{array}{ll}
& \lambda\,nil.M \equiv M\\
& \lambda\Sigma,\x.M \equiv\lambda\Sigma.(\lambda \x.M),
\end{array}
$\\[5pt] where $nil$ is the empty context. For example,
$$\lambda x,y,z.M\equiv\lambda xyz.M$$
\end{definition}
\begin{definition}(Free variables of judgements). By definition, put\\
$
\begin{array}{ll}
& FV(\Gamma\vdash M)=FV(\lambda\Gamma.M)\\
\end{array}
$
\end{definition}

\newpage

%% file: redukcii1.tex
\section{ $\lambda\pi$-calculus }
\begin{definition}\label{uparrow}We define  $\Uparrow_{\Delta}\!(s)$ as follows:
\\
$
\begin{array}{ll}
& \Uparrow_{nil}(s) \equiv s\\
& \Uparrow_{\Sigma,\x}(s)
\equiv\langle\pi_{\x}\,\ci\Uparrow_{\Sigma}\!(s)\mi \x\is
\x\rangle,
\end{array}
$\\[5pt] where $nil$ is the empty context. For example,
$$\Uparrow_{x,y,z}\!(s)\equiv\langle\pi_{z}\ci\langle\pi_{y}\ci\langle\pi_{x}\ci
s\mi x\is x\rangle\mi y\is y\rangle\mi z\is z\rangle$$
\end{definition}

Note that
$\Uparrow_{\Sigma,\x}\!(s)\equiv\,\Uparrow_{\x}\!\!(\Uparrow_{\Sigma}\!(s))$.
\begin{convention}$ $\\
$
\begin{array}{lll}
\Uparrow_{\Delta}\!\la s\mi N\is \x\ra & \text{is shorthand for} &
\Uparrow_{\Delta}\!(\la s\mi N\is \x\ra)
\end{array}
$
\end{convention}

Now we introduce several one-step reductions: two reductions with
the same name $\arrow$ defined on the sets of terms and
substitutions, and the reduction $\ri$ defined on the set of
judgements of the form $\Gamma\vdash M$.
\begin{definition}\label{lambdapi}(The calculus $\lambda\pi$).\\[10pt]
$
\begin{array}{ccc}
& \ruleone{ M_1\arrow M_2}{ \lambda
\x.M_1\arrow\lambda \x.M_2}\\[20pt]
& \ruleone{ M_1\arrow M_2}{ M_1N\arrow  M_2N} & \ruleone{
N_1\arrow N_2}{
MN_1\arrow  MN_2}\\[20pt]
& \ruleone{  s_1\arrow s_2}{s_1\ci M\arrow s_2\ci M} & \ruleone{
M_1\arrow
M_2}{s\ci M_1\arrow s\ci M_2}\\[20pt]
& \ruleone{ s_1\arrow s_2}{ \langle s_1\mi N\is \x\rangle\arrow
\langle s_2\mi N\is \x\rangle} & \ruleone{ N_1\arrow N_2}{ \langle
s\mi N_1\is \x\rangle\arrow \langle
s\mi N_2\is \x\rangle}\\[20pt]
& \ruleone{ s_1\arrow s_2}{ s_1\ci q\arrow s_2\ci q} & \ruleone{
q_1\arrow q_2}{ s\ci q_1\arrow s\ci
q_2}\\[20pt]
& \ruleone{M_1\arrow M_2}{\Gamma\vdash M_1\ri
\Gamma\vdash M_2}\\[20pt]
\end{array}
$

$
\begin{array}{lll}
(Beta) & (\lambda \x.M)N\rightarrow \langle id\mi N\is \x\rangle\ci M &\\
(Abs) & s\ci \lambda \x.M\rightarrow\lambda \x.\langle \pi_{\x}\ci
s\mi \x\is
\x\rangle\ci M &\\
(App) & s\ci MN\rightarrow(s\ci  M)(s\ci  N) &\\
(ConsVar) & \langle s\mi N\is \x\rangle\ci \x\rightarrow N &\\
(New) & \langle s\mi N\is \x\rangle\ci \y\rightarrow  s\ci \y & (\x\neq \y)\\
(IdVar) &  id\ci \x\rightarrow \x &\\
(Clos) & s\ci q\ci M\rightarrow (s\ci q)\ci M &\\
(Ass) & s\ci q\ci  r\rightarrow (s\ci q)\ci r &\\
(IdR) & s\ci id\rightarrow s &\\
(IdShift) & id\ci \pi_{\x}\rightarrow\pi_{\x} &\\
(ConsShift) & \langle s\mi N\is \x\rangle\ci \pi_{\x}\rightarrow s &\\
(Map) & s\ci \langle q\mi N\is \x\rangle\rightarrow\langle
s\ci q \mi s\ci N\is \x\rangle &\\
(\pi_1) & \pi_{\x}\ci \y\rightarrow \y & (\x\neq \y)\\
(\pi_2) & (s\ci\pi_{\x})\ci \y\rightarrow s\ci \y & (\x \neq
\y)\\
(\alpha_1) & \lambda \x.M\rightarrow\lambda
\y.\langle\pi_{\y}\mi \y\is \x\rangle\ci M & (*)\\
(\alpha_2) & \Gamma,\x,\Delta\vdash M\ri\Gamma,\y,\Delta\vdash\,
\Uparrow_{\Delta}\!\langle\pi_{\y}\mi \y\is \x\rangle\ci M & (**)
\\[5pt]
\end{array}$\\
where the side conditions are as follows:\\[5pt]
$
\begin{array}{ll}
(*) & \x \in \bigcup_{i\geqslant 1} FV_i(\lambda \x.M);\quad
\y\notin \bigcup_{i\geqslant 1} FV_i(\lambda \x.M
 )\\
 (**) & \x\in \bigcup_{i\geqslant 1} FV_i(\x,\Delta\vdash M);\quad \y\notin
 \bigcup_{i\geqslant 1}FV_i(\x,\Delta\vdash M
 )\\
\end{array}\\[5pt]
$
Recall that $s\ci MN$ is shorthand for $s\ci(MN)$.\\
\end{definition}
Note that $(Abs)$ can be written as\\
$
\begin{array}{ll}
(Abs) & s\ci \lambda \x.M\rightarrow\lambda
\x.\!\Uparrow_{\x}\!(s) \ci M
\end{array}
$
\begin{definition}By $\twoh$ denote the reflexive transitive
closure of $\arrow$.\\
By $\ri\ri$ denote the reflexive transitive closure of $\ri$.
\end{definition}
\begin{lemma}$ $\\
If $\x\neq \y_1, \ldots, \x\neq \y_k$,
then
 $\la s\mi N\is \x\mi N_1\is \y_1\mi\ldots\mi N_k\is \y_k\ra\ci
\x\twoh N$.
\end{lemma}
\begin{proof}We use $New$ (repeatedly), then we use $ConsVar$.
\end{proof}
\begin{example}$$\begin{array}{ll}
\langle id\mi N\is x\mi L\is y\rangle\ci x &
\\
\rightarrow\langle id\mi N\is x\rangle\ci x &
New\\
\rightarrow N & ConsVar
\end{array}$$
\end{example}
\begin{example}$$\begin{array}{ll}
\langle id\mi N\is x\mi L\is y\rangle\ci y &
\\
\rightarrow L & ConsVar
\end{array}$$
\end{example}
\begin{example}$$\begin{array}{ll}
\langle id\mi N\is x\mi L\is y\rangle\ci z &
\\
\rightarrow\langle id\mi N\is x\rangle\ci z &
New\\
\rightarrow id\ci z & New\\
\rightarrow z & IdVar
\end{array}$$
\end{example}
\begin{example}$$\begin{array}{ll}
\langle id\mi N\is x\mi L\is x\rangle\ci x &
\\
\rightarrow L & ConsVar
\end{array}$$
\end{example}
\begin{example}$$\begin{array}{ll}
\langle id\mi N\is x\mi L\is x\rangle\ci \pi_x\ci x &
\\
\arrow(\langle id\mi N\is x\mi L\is x\rangle\ci \pi_x)\ci x &
Clos\\
\arrow \langle id\mi N\is x\ra\ci x &
ConsShift\\
\rightarrow N & ConsVar
\end{array}$$
\end{example}
\begin{example}$$\begin{array}{ll}
\langle \pi_x\mi N\is y\mi L\is y\rangle\ci z &
\\
\rightarrow\langle \pi_x\mi N\is y\rangle\ci z &
New\\
\rightarrow\pi_x\ci z & New\\
\rightarrow z & \pi_1
\end{array}$$
\end{example}
\begin{example}$$\begin{array}{ll}
\langle\pi_x\mi N\is y\mi L\is y\rangle\ci x &
\\
\rightarrow\langle\pi_x\mi N\is y\rangle\ci
x & New\\
\rightarrow\pi_x\ci  x & New \end{array}$$
 where $\pi_x\ci x$ is a
normal form.
\end{example}
\begin{example}
$FV(\lambda x.\pi_x\ci x)=\la\{x\},\emptyset,\emptyset,\ldots\ra$
\end{example}
\begin{example}$$\begin{array}{ll}
 \lambda x.\pi_x\ci x & \\
 \rightarrow\lambda
y.\langle\pi_y\mi y\is x\rangle\ci \pi_x\ci x &
\alpha_1\\
\rightarrow\lambda y.(\langle\pi_y\mi y\is x\rangle\ci \pi_x)\ci x
& Clos\\
 \rightarrow\lambda y.\pi_y\ci
x & ConsShift\\
\rightarrow\lambda y.x & \pi_1
\end{array}$$
\end{example}
\begin{example}$$\begin{array}{ll}
(\lambda x y.x)\,y & \\
\rightarrow\langle id\mi y\is x\rangle\ci \lambda y.x &
Beta\\
\rightarrow\lambda y.\langle\pi_y\ci \langle id\mi y\is
x\rangle\mi y\is y\rangle\ci x & Abs\\
 \rightarrow\lambda
y.(\pi_y\ci \langle id\mi y\is x\rangle)\ci x &
New\\
\rightarrow\lambda y.\langle\pi_y\ci id\mi \pi_y\ci y\is
x\rangle\ci x & Map\\
 \rightarrow \lambda y.\pi_y\ci
y & ConsVar\\
\rightarrow\lambda z.\langle\pi_z\mi z\is y\rangle\ci \pi_y\ci y &
\alpha_1\\
\rightarrow\lambda z.(\langle\pi_z\mi z\is y\rangle\ci \pi_y)\ci y
& Clos\\
 \rightarrow\lambda z.\pi_z\ci
y & ConsShift\\
\rightarrow\lambda z.y & \pi_1
\end{array}
$$
\end{example}
\begin{example}$$\begin{array}{ll}
(\lambda x y.x)\,y & \\
\rightarrow \langle id\mi y\is x\rangle\ci\lambda y.x &
Beta\\
\rightarrow\lambda y.\langle\pi_y\ci\langle id\mi y\is x\rangle\mi
y\is y\rangle\ci x & Abs\\
 \rightarrow\lambda
y.\langle\pi_y\ci id\mi \pi_y\ci y\is x\mi y\is y\rangle\ci x &
Map\\
\rightarrow\lambda y.\langle\pi_y\ci id\mi
\pi_y\ci y\is x\rangle\ci x & New\\
 \rightarrow
\lambda y.\pi_y\ci y & ConsVar\\
\rightarrow\lambda z.\langle\pi_z\mi z\is y\rangle\ci \pi_y\ci y &
\alpha_1\\
\rightarrow\lambda z.(\langle\pi_z\mi z\is y\rangle\ci \pi_y)\ci y
 & Clos\\
 \rightarrow\lambda z.\pi_z\ci
y & ConsShift\\
\rightarrow\lambda z.y & \pi_1
\end{array}$$
\end{example}
\begin{example}
$FV(x\vdash\pi_x\ci x)=FV(\lambda x.\pi_x\ci
x)=\la\{x\},\emptyset,\emptyset,\ldots\ra$
\end{example}
\begin{example}
$$ x,x\vdash \pi_x\ci x \ri_{\alpha_2} x,y\vdash\langle\pi_y\mi
y\is x\rangle\ci \pi_x\ci x$$
 Further,
 $$\begin{array}{ll}
 \langle \pi_y\mi y\is
x\rangle\ci \pi_x\ci x & \\
\rightarrow (\langle \pi_y\mi y\is x\rangle\ci \pi_x)\ci x &
Clos\\
\rightarrow \pi_y\ci x & ConsShift\\
\rightarrow x & \pi_1 \end{array}$$
 We see that
$$\begin{array}{ll}
 x,x\vdash \pi_x\ci
x & \\
\ri x,y\vdash\langle\pi_y\mi y\is x\rangle\ci \pi_x\ci x &
\alpha_2\\
\ri x,y\vdash (\langle \pi_y\mi y\is x\rangle\ci \pi_x)\ci x &
Clos\\
\ri x,y\vdash \pi_y\ci x & ConsShift\\
\ri x,y\vdash x & \pi_1
\end{array}$$
\end{example}
\begin{example}
$FV(\lambda x z.\pi_z\ci\pi_x\ci
x)=\la\{x\},\emptyset,\emptyset,\ldots\ra$
\end{example}
\begin{example}$$\begin{array}{ll}
\lambda x x z.\pi_z\ci\pi_x\ci x & \\
\arrow \lambda x y.\la\pi_y\mi y\is
x\ra\ci\lambda z.\pi_z\ci\pi_x\ci x & \alpha_1\\
\arrow
 \lambda x y
z.\Uparrow_z\!\la\pi_y\mi y\is x\ra\ci\pi_z\ci\pi_x\ci x & Abs\\
\equiv\lambda x y z.\la\pi_z\ci\la\pi_y\mi y\is x\ra\mi z\is
z\ra\ci\pi_z\ci\pi_x\ci x & Definition~\ref{uparrow}\\
\arrow \lambda x y z.(\la\pi_z\ci\la\pi_y\mi y\is x\ra\mi z\is
z\ra\ci\pi_z)\ci\pi_x\ci x & Clos\\
\arrow
\lambda x y z.(\pi_z\ci\la\pi_y\mi y\is x\ra)\ci\pi_x\ci x & ConsShift\\
\arrow
\lambda x y z.\la\pi_z\ci\pi_y\mi\pi_z\ci y\is x\ra\ci\pi_x\ci x & Map\\
\arrow
\lambda x y z.(\la\pi_z\ci\pi_y\mi\pi_z\ci y\is x\ra\ci\pi_x)\ci x & Clos\\
\arrow
\lambda x y z.(\pi_z\ci\pi_y)\ci x & ConsShift\\
\arrow
\lambda x y z.\pi_z\ci x & \pi_2\\
\arrow \lambda x y z. x & \pi_1
\end{array}$$
\end{example}
\begin{example}
$FV(x,z\vdash\pi_z\ci\pi_x\ci x)=FV(\lambda x z.\pi_z\ci\pi_x\ci
x)=\la\{x\},\emptyset,\emptyset,\ldots\ra$
\end{example}
\begin{example}$$\begin{array}{ll}
x,x,z\vdash\pi_z\ci\pi_x\ci x & \\
\ri x,y,z\vdash\,\Uparrow_z\!\la\pi_y\mi y\is
x\ra\ci\pi_z\ci\pi_x\ci x & \alpha_2\\
\equiv x,y,z\vdash\la\pi_z\ci\la\pi_y\mi y\is x\ra\mi z\is
z\ra\ci\pi_z\ci\pi_x\ci x & Definition~\ref{uparrow}\\
\ri x,y,z\vdash(\la\pi_z\ci\la\pi_y\mi y\is x\ra\mi z\is
z\ra\ci\pi_z)\ci\pi_x\ci x & Clos\\
\ri
x,y,z\vdash(\pi_z\ci\la\pi_y\mi y\is x\ra)\ci\pi_x\ci x & ConsShift\\
\ri
x,y,z\vdash\la\pi_z\ci\pi_y\mi\pi_z\ci y\is x\ra\ci\pi_x\ci x & Map\\
\ri
x,y,z\vdash(\la\pi_z\ci\pi_y\mi\pi_z\ci y\is x\ra\ci\pi_x)\ci x & Clos\\
\ri
x,y,z\vdash(\pi_z\ci\pi_y)\ci x & ConsShift\\
\ri
x,y,z\vdash\pi_z\ci x & \pi_2\\
\ri x,y,z\vdash x & \pi_1 \end{array}$$
\end{example}
\begin{example}(Some redexes are underlined).\\
$$\begin{array}{ll} (\lambda x y z.xz(yz))(\lambda x
y.x) & \\
\arrow \la id\mi \lambda x y.x\is x\ra\ci\lambda
y z.xz(yz) & Beta\\
\arrow \lambda y.\Uparrow_{y}\!\la id\mi \lambda x
y.x\is x\ra\ci\lambda z. xz(yz) & Abs\\
 \arrow \lambda y z.\Uparrow_{y,z}\!\!\la
id\mi \lambda x y.x\is x\ra\ci xz(yz) & Abs\\
\arrow \lambda y z.(\underline{\Uparrow_{y,z}\!\!\la id\mi \lambda
x y.x\is x\ra\ci xz})\,(\Uparrow_{y,z}\!\!\la id\mi \lambda x
y.x\is x\ra\ci yz) &  App\\
\twoh \lambda y z.(\lambda x
y.x)\,z\,(\underline{\Uparrow_{y,z}\!\!\la id\mi \lambda x
y.x\is x\ra\ci yz}) &  App, Example~\ref{4.25},Example~\ref{4.26}\\
\twoh \lambda y z.(\lambda x
y.x)\,z\,(yz) & App,Example~\ref{4.27},Example~\ref{4.26}\\
\arrow \lambda y z.(\la id\mi z\is
x\ra\ci\lambda y.x)(yz) & Beta\\
\arrow \lambda y z.(\lambda
y.\underline{\Uparrow_y\!\la id\mi z\is x\ra\ci x})(yz) & Abs\\
\twoh \lambda y
z.(\lambda y. z)(yz) & Example~\ref{4.28}\\
\arrow \lambda y z.\la id\mi
yz\is y\ra\ci z & Beta\\
\arrow \lambda y z.id\ci
z & New\\
\arrow \lambda y z.z & IdVar
\end{array}$$
\end{example}

\begin{example}\label{4.25}
$$\begin{array}{ll}
\Uparrow_{y,z}\!\la id\mi \lambda xy.x\is x\ra\ci x & \\
\equiv\la\pi_z\ci\Uparrow_y\!\la id\mi \lambda xy.x\is x\ra\mi
y\is y\ra\ci x &
Definition~\ref{uparrow}\\
\arrow(\pi_z\ci\Uparrow_y\! \la id\mi \lambda xy.x\is x\ra)\ci x &
New\\
\equiv(\pi_z\ci\la\pi_y\ci \la id\mi \lambda xy.x\is x\ra\mi y\is
y\ra)\ci x &
Definition~\ref{uparrow}\\
\arrow\la\pi_z\ci\pi_y\ci\la id\mi \lambda xy.x\is x\ra\mi
\pi_z\ci y\is y\ra\ci x
& Map\\
\arrow(\pi_z\ci\pi_y\ci\la id\mi \lambda xy.x\is x\ra)\ci x &
New\\
\arrow(\pi_z\ci\la\pi_y\ci id\mi \pi_y\ci\lambda xy.x\is x\ra)\ci
x & Map\\
\arrow\la\pi_z\ci\pi_y\ci id\mi\pi_z\ci\pi_y\ci\lambda xy.x\is
x\ra\ci x & Map\\
\arrow\pi_z\ci\pi_y\ci\lambda xy.x & ConsVar\\
\twoh \lambda x y.x & because\,\, \lambda xy.x\,\, is\,\, closed
\end{array}$$
\end{example}
\begin{example}\label{4.26}
$$\begin{array}{ll}
\Uparrow_{y,z}\!\la id\mi \lambda xy.x\is x\ra\ci z & \\
\equiv\la\pi_z\ci \Uparrow_y\!\la id\mi \lambda xy.x\is x\ra\mi z\is z\ra\ci z & Definition~\ref{uparrow}\\
\arrow z & ConsVar
\end{array}$$
\end{example}
\begin{example}\label{4.27}
$$\begin{array}{ll}
\Uparrow_{y,z}\!\la id\mi\lambda xy.x\is x\ra\ci y & \\
\equiv\la\pi_z\ci\Uparrow_y\!\la id\mi \lambda xy.x\is x\ra\mi
z\is
z\ra\ci y & Definition~\ref{uparrow}\\
\arrow(\pi_z\ci\Uparrow_y\!\la id\mi \lambda xy.x\is x\ra)\ci y &
New\\
\equiv(\pi_z\ci\la\pi_y\ci\la id\mi\lambda xy.x\is x\ra\mi y\is
y\ra)\ci y & Definition~\ref{uparrow}\\
\arrow\la\pi_z\ci\pi_y\ci\la id\mi \lambda xy.x\is x\ra\mi
\pi_z\ci
y\is y\ra\ci y & Map\\
\arrow \pi_z\ci y & ConsVar\\
\arrow y & \pi_1
\end{array}$$
\end{example}
\begin{example}\label{4.28}
$$\begin{array}{ll}
\Uparrow_y\!\la id\mi z\is x\ra\ci x & \\
\equiv\la\pi_y\ci\la id\mi z\is x\ra\mi y\is y\ra\ci x &
Definition~\ref{uparrow}\\
\arrow(\pi_y\ci\la id\mi z\is x\ra)\ci x & New\\
\arrow\la\pi_y\ci id\mi\pi_y\ci z\is x\ra\ci x & Map\\
\arrow\pi_y\ci z & ConsVar\\
\arrow z & \pi_1
\end{array}$$
\end{example}

\newpage

%% file: subjectreduction.tex
\section{Subject reduction}
\begin{theorem}[Subject reduction, part one]\label{SubRed1}$ $\\
If $\Gamma\vdash M_1$ is derivable and $M_1\arrow M_2$, then
$\Gamma\vdash M_2$ is derivable.\\
If $\Gamma\vdash s_1\tri\Delta$ is derivable and $s_1\arrow s_2$,
then $\Gamma\vdash s_2\tri\Delta$ is derivable.
 \end{theorem}
 \begin{proof}The proof is straightforward, but tedious.\\[5pt]
 Case $Beta. \quad (\lambda \x.M)N\rightarrow \langle id\mi
N\is \x\rangle\ci M
$\\[5pt]
$$\ruletwo{\ruleone{\ruledot{\Gamma,\x\vdash
M}}{\Gamma\vdash\lambda
\x.M}}{\ruledot{\Gamma\vdash N}}{\Gamma\vdash (\lambda \x.M)N}$$\\[5pt]
$$\ruletwo{\ruletwo{\Gamma\vdash
id\tri\Gamma}{\ruledot{\Gamma\vdash N}}{\Gamma\vdash\la id\mi N\is
\x\ra\tri\Gamma,\x}}{\ruledot{\Gamma,\x\vdash
M}}{\Gamma\vdash\la id\mi N\is \x\ra\ci M}$$\\[5pt]
Case $Abs. \quad s\ci \lambda \x.M\rightarrow\lambda \x.\langle
\pi_{\x}\ci s\mi \x\is \x\rangle\ci M
$\\[5pt]
$$\ruletwo{\ruledot{\Gamma\vdash
s\tri\Delta}}{\ruleone{\ruledot{\Delta,\x\vdash
M}}{\Delta\vdash\lambda \x.M}}{\Gamma\vdash s\ci\lambda
\x.M}$$\\[5pt]
$$\ruleone{\ruletwo{\ruletwo{\ruletwo{\Gamma,\x\vdash\pi_{\x}\tri\Gamma}{\ruledot{\Gamma\vdash
s\tri\Delta}}{\Gamma,\x\vdash\pi_{\x}\ci
s\tri\Delta}}{\Gamma,\x\vdash \x}{\Gamma,\x\vdash\la\pi_{\x}\ci
s\mi \x\is \x\ra\tri\Delta,\x}}{\ruledot{\Delta,\x\vdash
M}}{\Gamma,\x\vdash\la\pi_{\x}\ci s\mi \x\is \x\ra\ci
M}}{\Gamma\vdash\lambda \x.\la\pi_{\x}\ci s\mi \x\is
\x\ra\ci M}$$\\[5pt]
Case $\alpha_1$. $ \lambda \x.M\rightarrow\lambda
\y.\langle\pi_{\y}\mi \y\is \x\rangle\ci M $ $\quad (*)$ \\[5pt]
$$\ruleone{\ruledot{\Gamma,\x\vdash M}}{\Gamma\vdash\lambda \x.M}$$\\[5pt]
$$\ruleone{\ruletwo{\ruletwo{\Gamma,\y\vdash\pi_{\y}\tri\Gamma}{\Gamma,\y\vdash
\y}{\Gamma,\y\vdash\la\pi_{\y}\mi \y\is
\x\ra\tri\Gamma,\x}}{\ruledot{\Gamma,\x\vdash
M}}{\Gamma,\y\vdash\la\pi_{\y}\mi \y\is \x\ra\ci
M}}{\Gamma\vdash\lambda \y.\la\pi_{\y}\mi \y\is \x\ra\ci
M}$$\\[5pt]
 And so on. Note that we do not use $(*)$ in the proof of the case $\alpha_1$.
 \end{proof}
 \begin{lemma}\label{lemma1.2}If $\Gamma\vdash s\ci\lambda \x.M$ is
 derivable, then $\Gamma\vdash\lambda
\x.\!\Uparrow_{\x}\!(s) \ci M$ is
 derivable.
 \end{lemma}
 \begin{proof}Theorem~\ref{SubRed1}, the case $Abs$.
 \end{proof}
 \begin{lemma}\label{lemma31}If $\Gamma\vdash s\ci\lambda\Delta.M$ is
derivable, then $\Gamma\vdash\lambda
\Delta.\!\Uparrow_{\Delta}\!(s)\ci M$ is derivable.
\end{lemma}
\begin{proof}Recall that
$\Uparrow_{\Sigma,\x}\!(s)\equiv\,\Uparrow_{\x}\!\!(\Uparrow_{\Sigma}\!(s))$.
Now we can use Lemma~\ref{lemma1.2} repeatedly.
\end{proof}
 \begin{theorem}[Subject reduction, part two]$ $\\
 Suppose\\
  $\Gamma,\x,\Delta\vdash M$ is derivable and\\
  $\Gamma,\x,\Delta\vdash M\ri_{\alpha_2} \Gamma,\y,\Delta\vdash\,
\Uparrow_{\Delta}\!\langle\pi_{\y}\mi \y\is \x\rangle\ci M$;\\
then\\ $\Gamma,\y,\Delta\vdash\,
\Uparrow_{\Delta}\!\langle\pi_{\y}\mi \y\is \x\rangle\ci M$ is
derivable.
\end{theorem}
\begin{proof}By Generation lemma $\Gamma,\x,\Delta\vdash M$ is derivable iff
$\Gamma\vdash\lambda \x.\lambda\Delta.M$ is derivable. To conclude
the proof, it is sufficient to prove the following lemma.
\end{proof}
\begin{lemma}Suppose\\
 $\Gamma\vdash\lambda \x.\lambda\Delta.M$ is
derivable; then\\ $\Gamma\vdash\lambda
\y.\lambda\Delta.\Uparrow_{\Delta}\!\langle\pi_{\y}\mi \y\is
\x\rangle\ci M$ is derivable.
\end{lemma}
\begin{proof}
We use Theorem~\ref{SubRed1} (the case $\alpha_1$) and
 Lemma~\ref{lemma31}.
\end{proof}
\newpage

%% file: normalforms.tex
\section{Two theorems about normal forms}
\begin{definition}By $\sigma\pi\alpha$ denote $\lambda\pi$ without
$Beta$.
\end{definition}
\begin{definition}By $\pi_{\x_1\ldots\x_n}$ denote
$(\ldots((\pi_{\x_1}\ci\pi_{\x_2})\ci\ldots)\ci\pi_{\x_n}$
\end{definition}
\begin{theorem}\label{Subnf}Suppose $\Gamma\vdash s\tri\Delta$ is derivable
and $s$ is a $\sigma\pi\alpha$-normal form (with respect to
$\arrow$); then $s$ has one of the following forms:\\ $
\begin{array}{lll}
(i) & id &\\
(ii) & \pi_{\x_1\ldots\x_n} & (n\geqslant 1)\\
(iii) & \la id\mi N_1\is \y_1\mi\ldots\mi N_k\is \y_k\ra &
(k\geqslant
1)\\
(iv) & \la\pi_{\x_1\ldots\x_n}\mi N_1\is \y_1\mi\ldots\mi N_k\is
\y_k\ra & (n\geqslant 1, k\geqslant 1)
\end{array}
$\\
\end{theorem}
Of course, the terms $N_1,\ldots,N_k$ are not arbitrary, they are
$\sigma\pi\alpha$-normal forms (with respect to $\arrow$).
\begin{proof}The proof is by induction over the structure of $s$ (see Definition~\ref{pervoe}). The set of substitutions of the forms $(i)-(iv)$ contains $id$ and
$\pi_{\x}$ for any $\x$. This set  is also closed under
 $\la - \mi N\is \y\ra$ for any $N,\y$. To conclude the proof, it is sufficient to prove the
 following
lemma.
\end{proof}
\begin{lemma}If $\Gamma\vdash s\ci q\tri\Delta$ is derivable and
both $s,q$ belong to $(i),(ii),(iii),(iv)$, then $s\ci q\,$
$\sigma\pi\alpha$-reduces to one of the forms
$(i),(ii),(iii),(iv)$.
\end{lemma}
\begin{proof}Let us considered five cases.
\\[5pt]
Case 1: $q$ is $id$.\\
$s\ci id\arrow s$\\[5pt]
 Case 2: $s$ is $id$ and $q$ has the form
 $\pi_{\z_1\ldots\z_m}$.\\
  $id\ci
\pi_{\z_1\ldots\z_m}\twoh
 \pi_{\z_1\ldots\z_m}$\\[5pt]
Case 3: $s$ has the form $\pi_{\x_1\ldots\x_n}$ and $q$
has the form $\pi_{\z_1\ldots\z_m}$.\\
$(\pi_{\x_1\ldots\x_n})\ci\pi_{\z_1\ldots\z_m}\twoh
\pi_{\x_1\ldots\x_n\z_1\ldots\z_m}$\\[5pt]
Case 4: $s$ has the form $\la r\mi N_1\is \y_1\mi\ldots\mi N_k\is
\y_k\ra$ and $q$ has the form $\pi_{\z_1\ldots\z_m}$, where $r$ is
$id$ or
$\pi_{\x_1\ldots\x_n}$. Hence, $s\circ q$ has the form\\
$\la r\mi N_1\is \y_1\mi\ldots\mi N_k\is
\y_k\ra\ci\pi_{\z_1\ldots\z_m}$.
 By Generation lemma, $\y_k= \z_1$,\\ $\y_{k-1}= \z_2$, and so
on (see Example~\ref{bedsubst}). \\[5pt]
If $k=m$,\,then\\[5pt]
$\la r\mi N_1\is \y_1\mi\ldots\mi N_k\is
\y_k\ra\ci\pi_{\z_1\ldots\z_m}$\quad is the same as\\
$\la r\mi N_1\is \y_1\mi\ldots\mi N_k\is
\y_k\ra\ci\pi_{\y_k\ldots\y_1}\twoh r$\\[5pt]
If $k>m$,\,then \\[5pt]
$\la r\mi N_1\is \y_1\mi\ldots\mi N_k\is
\y_k\ra\ci\pi_{\z_1\ldots\z_m}$\quad is the same as\\
$\la r\mi N_1\is \y_1\mi\ldots\mi N_k\is
\y_k\ra\ci\pi_{\y_k\ldots\y_{k-m+1}}\twoh \la r\mi N_1\is
\y_1\mi\ldots\mi N_{k-m}\is \y_{k-m}\ra $\\[5pt]
If $k<m$,\,then\\[5pt]
$\la r\mi N_1\is \y_1\mi\ldots\mi N_k\is
\y_k\ra\ci\pi_{\z_1\ldots\z_m}$\quad is the same as\\
$\la r\mi N_1\is \z_k\mi\ldots\mi N_k\is
\z_1\ra\ci\pi_{\z_1\ldots\z_m}$\\
  If $r$ is $id$, this term reduces to
  $\pi_{\z_{k+1}\ldots\z_m}$. If $r$ is
  $\pi_{\x_1\ldots\x_n}$, this term reduces to
  $\pi_{\x_1\ldots\x_n\z_{k+1}\ldots\z_m}$.\\[5pt]
 Case 5: $q$ has the form $\la r\mi N_1\is
\y_1\mi\ldots\mi N_k\is \y_k\ra$,
 where $r$ is $id$ or $\pi_{\z_1\ldots\z_m}$.\\
$s\ci\la r\mi N_1\is \y_1\mi\ldots\mi N_k\is \y_k\ra\twoh\la s\ci
r\mi s\ci N_1\is \y_1\mi\ldots\mi s\ci N_k\is \y_k\ra$\\
  Then we
use the previous cases to reduce $s\ci r$.
\end{proof}
Note that we do not use $\pi_1,\pi_2,\alpha_1,\alpha_2$ in this
proof.
\begin{definition}A term $M$ is called \emph{pure} iff it does not
contain sub-terms of the shape $s\ci N$.
\end{definition}
\begin{theorem}
If $\Gamma\vdash M$ is derivable and $\Gamma\vdash M$ is a
$\sigma\pi\alpha$-normal form (with respect to $\ri$), then $M$ is
pure.\label{pure}
\end{theorem}
\begin{proof}
Suppose $M$ contain a sub-term of the shape $s\ci N$; then $N$
must be a variable (we  denote it by $\y$), else we can apply
$Abs, App$ or $Clos$. The substitution $s$ is a
$\sigma\pi\alpha$-normal form  and must have the form
$\pi_{\x_1\ldots\x_n}$ (see theorem~\ref{Subnf}), else we can
apply $IdVar, ConsVar$ or $New$. Further,  $\x_n$ in
$(\pi_{\x_1\ldots\x_n})\ci \y$ must coincide with $\y$, else we
can apply $\pi_1$ or $\pi_2$. We see that $M$ must be constructed
from  variables and  blocks of the form
 $(\pi_{\x_1\ldots\x_m\y})\ci \y\quad (m\geqslant 0)$ by
using application and abstraction. To conclude the proof, it is
sufficient to prove the following lemma.
\end{proof}
\begin{lemma}
If $\Gamma\vdash M$ is derivable and $M$ is constructed from
variables and  blocks of the form $(\pi_{\x_1\ldots\x_m\y})\ci
\y\quad(m\geqslant 0)$ by using application and abstraction, then
$M$ is pure (this means that $M$ does not contain blocks) or we
can apply $\alpha_1$ or $\alpha_2$ to $\Gamma\vdash M$.
\end{lemma}
\begin{proof}The proof is by induction over the structure of $M$. Let us consider four cases.\\[5pt]
Case 1: $M$ is a variable. The proof is trivial.\\[5pt]
 Case 2:  $ M$
has the form $(\pi_{\x_1\ldots\x_m\y})\ci \y$. By Generation
lemma, $\Gamma\vdash M$
has the form\\
$\Delta\mi\y\mi\x_m\mi\ldots\mi \x_1\vdash
(\pi_{\x_1\ldots\x_m\y})\ci \y$  \\
and we can apply $\alpha_2$ , because\\
 $FV(\y\mi\x_m\mi\ldots\mi
\x_1\vdash (\pi_{\x_1\ldots\x_m\y})\ci \y)=\la\{\y\},\emptyset,\emptyset,\ldots\ra$\\[10pt]
 Case 3: $M$ has the form $\lambda \x.N$. By Generation lemma, $\Gamma,\x\vdash
N$ is derivable. Suppose $N$ contains a block of the form
 $(\pi_{\x_1\ldots\x_m\y})\ci \y$. By
induction hypothesis, we can apply $\alpha_1$ or $\alpha_2$ to
$\Gamma,\x\vdash N$. But any application  of $\alpha_1$ or
$\alpha_2$ to $\Gamma,\x\vdash N$ corresponds to some application
of $\alpha_1$ or $\alpha_2$ (and, in some cases, $Abs$) to
$\Gamma\vdash\lambda \x.N$. For example,\\
$\Gamma,\x\vdash N\ri_{\alpha_2}\Gamma,\y \vdash\la\pi_{\y}\mi
\y\is \x\ra\ci N$\\
corresponds to\\
 $\Gamma\vdash\lambda
\x.N\ri_{\alpha_1}\Gamma\vdash\lambda \y.\la\pi_{\y}\mi
\y\is \x\ra\ci N$\\[10pt]
 Case 4: $M$ has the form $NL$. By Generation lemma, $\Gamma\vdash N$ and
$\Gamma\vdash L$ are derivable. Suppose one of these terms
contains a block of the form $(\pi_{\x_1\ldots\x_m\y})\ci \y$. For
clarity, let it be $N$. By induction hypothesis,  we can apply
$\alpha_1$ or $\alpha_2$ to $\Gamma\vdash N$. I claim that we can
apply $\alpha_1$ or $\alpha_2$ to $\Gamma\vdash NL$. For the case
$\alpha_1$ is nothing to prove, because any $\alpha_1$-redex in
$N$ occurs in $NL$ too. For the case $\alpha_2$, suppose that
$\Gamma\vdash NL$ has the form $\Sigma,\x,\Delta\vdash NL$. Recall
that $FV(NL)=FV(N)\cup FV(L)$, hence $FV(N)\subseteq FV(NL)$. By
Lemma~\ref{monotone},\\ $FV(\x,\Delta\vdash N))\subseteq FV(\x,\Delta\vdash NL)$.\\
Hence if $\x\in \bigcup_{i\geqslant 1} FV_i(\x,\Delta\vdash N)$,
then $\x\in \bigcup_{i\geqslant 1} FV_i(\x,\Delta\vdash NL)$. If
we can apply $\alpha_2$ to $\Sigma,\x,\Delta\vdash N$, then we can
apply $\alpha_2$ to $\Sigma,\x,\Delta\vdash NL$.
\end{proof}
 Warning!
In general, the terms $N_1,\ldots,N_k$ in the statement of
Theorem~\ref{Subnf} are not pure.  For example, the judgement\\
$x,x\vdash\la id\mi\pi_x\ci x\is y\ra\tri x,x,y$\\ is derivable
and the substitution\\ $\la id\mi\pi_x\ci x\is y\ra$\\ is a
$\sigma\pi\alpha$-normal form.
\newpage

%% file: lambda-sigma.tex
\section{Relation with $\lambda\sigma$}
\begin{definition} The symbols $U,V,W$ range over \emph{name-free
terms} and the symbols $u,v,w$  range over \emph{name-free
substitutions}. The sets of name-free terms and name-free
substitutions are defined inductively as follows:
\begin{align*}
U,V::&= 1 \mid UV \mid \lambda U \mid u\ci U \\
u,v::&= id \mid \pi \mid \langle u\mi V \rangle \mid u\ci v
\end{align*}
\end{definition}
\begin{definition}(The calculus $\lambda\sigma$ in the new notation).\\[10pt]
$
\begin{array}{ccc}
& \ruleone{ U_1\arrow U_2}{ \lambda
U_1\arrow\lambda U_2} &\\[20pt]
& \ruleone{ U_1\arrow U_2}{ U_1V\arrow U_2V} & \ruleone{ V_1\arrow
V_2}{
UV_1\arrow UV_2}\\[20pt]
& \ruleone{ u_1\arrow u_2}{u_1\ci U\arrow u_2\ci U} & \ruleone{
U_1\arrow
U_2}{u\ci U_1\arrow u\ci U_2}\\[20pt]
& \ruleone{ u_1\arrow u_2}{ \langle u_1\mi V \rangle\arrow \langle
u_2\mi V\rangle} & \ruleone{ V_1\arrow V_2}{ \langle u\mi V_1
\rangle\arrow \langle
u\mi V_2\rangle}\\[20pt]
& \ruleone{ u_1\arrow u_2}{ u_1\ci v\arrow u_2\ci v} & \ruleone{
v_1\arrow v_2}{ u\ci v_1\arrow u\ci
v_2}\\[20pt]
\end{array}
$

$
\begin{array}{lll}
(Beta) & (\lambda U)V\rightarrow \langle id\mi V\rangle\ci U &\\
(Abs) & u\ci \lambda U\rightarrow\lambda \langle \pi\ci u\mi
1\rangle\ci U &\\
(App) & u\ci UV\rightarrow(u\ci U)(u\ci  V) &\\
(ConsVar) & \langle u\mi V\rangle\ci 1\rightarrow V &\\
(IdVar) &  id\ci 1\rightarrow 1 &\\
(Clos) & u\ci v\ci V\rightarrow (u\ci v)\ci V &\\
(Ass) & u\ci v\ci  w\rightarrow (u\ci v)\ci w &\\
(IdR) & u\ci id\rightarrow u &\\
(IdShift) & id\ci \pi\rightarrow\pi &\\
(ConsShift) & \langle u\mi V\rangle\ci \pi\rightarrow u &\\
(Map) & u\ci \langle v\mi V\rangle\rightarrow\langle u\ci v \mi
u\ci V\rangle &
\end{array}
$\\[5pt]
\end{definition}
\begin{definition}By $\sigma$ denote $\lambda\sigma$ without $Beta$.\\
By $\sigma(U)$ denote the $\sigma$-normal form of $U$ (this normal
form exists and is uniquely defined because $\sigma$ is strongly
normalizing and confluent).
\end{definition}
\begin{definition}By definition, put\\
$\underline{n}\equiv\underbrace{((\ldots(\pi\ci\pi)\ci\ldots)\ci\pi)}_{\mbox{\scriptsize
$n-1$ times}}\ci 1 \qquad(n\geqslant 1)$
\end{definition}
We see that $\underline{1}\equiv 1$, $\underline{2}\equiv\pi\ci
1$, and $\underline{n+1}\equiv\sigma(\pi\ci\underline{n})$.
\begin{definition} $(\Gamma\vdash
M)\boldsymbol{\Rightarrow} U$ is shorthand for ``the name-free
term
$U$ corresponds to the judgement $\Gamma\vdash M$."\\
$(\Gamma\vdash s\tri\Delta)\boldsymbol{\Rightarrow} u$ is
shorthand for ``the name-free
 substitution
$u$ corresponds to the
 judgement $\Gamma\vdash s\tri\Delta$."
\end{definition}
\begin{definition}\label{transfer}(The rules of correspondence between judgements and
name-free
terms/substitutions).\\[5pt]
 $
\begin{array}{ll}
(i)& (\Gamma,\x\vdash \x)\boldsymbol{\Rightarrow}  1 \\[5pt]
 (ii)& \ruleone{(\Gamma\vdash
\x)\boldsymbol{\Rightarrow}  \underline{n}}{(\Gamma,\y\vdash \x)\boldsymbol{\Rightarrow} \underline{n+1}\using {\quad(\x\neq \y)}}\\[20pt]
 (iii)& \ruletwo{(\Gamma\vdash M)\boldsymbol{\Rightarrow}  U}{(\Gamma\vdash N)\boldsymbol{\Rightarrow } V}{(\Gamma\vdash
MN)\boldsymbol{\Rightarrow}  UV}\\[15pt]
 (iv)& \ruleone{(\Gamma,\x\vdash M)\boldsymbol{\Rightarrow }  U}{(\Gamma\vdash\lambda \x.M)\boldsymbol{\Rightarrow}  \lambda U}\\[15pt]
 (v)& \ruletwo{(\Gamma\vdash s \tri\Delta)\boldsymbol{\Rightarrow}  u }{(\Delta\vdash M)\boldsymbol{\Rightarrow }  U}{(\Gamma\vdash
 s \ci M)\boldsymbol{\Rightarrow }  u \ci U}\\[15pt]
(vi)& (\Gamma\vdash id \tri\Gamma)\boldsymbol{\Rightarrow}  id  \\[5pt]
(vii)& (\Gamma,\x\vdash\pi_{\x} \tri\Gamma)\boldsymbol{\Rightarrow}   \pi  \\[5pt]
(viii)& \ruletwo{(\Gamma\vdash s
\tri\Delta)\boldsymbol{\Rightarrow}
 u  }{(\Gamma\vdash
N)\boldsymbol{\Rightarrow}  V}{(\Gamma\vdash
\langle s\mi N\is \x\rangle \tri\Delta ,\x)\boldsymbol{\Rightarrow} \la u\mi V\rangle }\\[15pt]
(ix)& \ruletwo{(\Gamma\vdash s \tri\Delta)\boldsymbol{\Rightarrow}
 u }{(\Delta\vdash q
\tri\Sigma)\boldsymbol{\Rightarrow}  v  }{(\Gamma \vdash s \ci q
\tri\Sigma)\boldsymbol{\Rightarrow}  u \ci v }\\[20pt]
\end{array}
$
\end{definition}
\begin{example}
$$
(x\vdash x)\boldsymbol{\Rightarrow} 1
$$
\end{example}
\begin{example}
$$
\ruleone{(x\vdash x)\boldsymbol{\Rightarrow} 1}{(x,y \vdash
x)\boldsymbol{\Rightarrow}  \pi\ci 1}
$$
\end{example}
\begin{example}
$$
\ruletwo{(x,y \vdash \pi_y\tri x) \boldsymbol{\Rightarrow} \pi}{(x
\vdash x) \boldsymbol{\Rightarrow} 1}{(x,y \vdash \pi_y\ci x)
\boldsymbol{\Rightarrow}
  \pi\ci 1}
$$
\end{example}
\begin{corollary}
 $(\Gamma,\x,\y_1,\ldots,\y_n\vdash
\x)\boldsymbol{\Rightarrow}\underline{n+1}$ if $\x\neq
\y_1,\ldots,\x\neq \y_n$.
\end{corollary}
\begin{lemma}\label{poleznaya}If $(\Gamma\vdash M)\boldsymbol{\Rightarrow} U$, then
$\Gamma\vdash M$ is derivable.
\end{lemma}
\begin{proof}The proof is straightforward, see Definition~\ref{defsequentsder}
and Definition~\ref{transfer}.
\end{proof}
\begin{definition}We  write
$(\Gamma\vdash M)\backsimeq(\Delta\vdash N)$
iff\\
$(\Gamma\vdash M)\boldsymbol{\Rightarrow} U$ and\\
$(\Delta\vdash N)\boldsymbol{\Rightarrow} U$, for some $U$.
\end{definition}
\begin{example}
$(x,y\vdash x)\backsimeq (x,y\vdash\pi_y\ci x)$
\end{example}
\begin{definition}A name-free term $U$ is called \emph{pure} if
it is constructed from the terms $\underline{n}$ by using
application and abstraction.
\end{definition}
\begin{lemma}\label{purelemma}If $(\Gamma\vdash M)\boldsymbol{\Rightarrow} U$ and $
M$ is pure, then $U$ is pure. If  $U$ is pure, then $U$ is a
$\sigma$-normal form.
\end{lemma}
\begin{proof}Each pure term $M$ is constructed from variables by using application and abstraction.
\end{proof}
\begin{definition}By definition, put\\
$\Uparrow(s)\equiv\la\pi\ci s\mi 1\ra$ \\
$\Uparrow^n\!(s)\equiv\underbrace{\Uparrow(\ldots(\Uparrow}_{\mbox{\scriptsize
$n$ times}}(s))\ldots) $
\end{definition}
\begin{lemma}\label{lemmauparrow}For any nameless term $U$,\\
 $\sigma(\la\pi\mi 1\ra\ci
U)\equiv\sigma(U)$ and\\
$\sigma(\Uparrow^n\!\la\pi\mi 1\ra\ci U)\equiv\sigma(U)$
\end{lemma}
\begin{proof}See~\cite{Abadi}, Lemma 3.6.
\end{proof}
\begin{definition}By $\overset{T}{\ri\ri}$ and
$\overset{T}{\twoh}$ denote (reflexive and transitive) reductions
in a calculus $T$ ($T$ may be
$\lambda\pi,\sigma\pi\alpha,\lambda\sigma$ or $\sigma$).
\end{definition}
\begin{theorem}\label{maintheorem}Suppose \\
$\begin{array}{l}
 \Gamma\vdash M\overset{\sigma\pi\alpha}{\ri\ri}\Sigma\vdash L;\\
 \Sigma\vdash L \text{ is a $\sigma\pi\alpha$-normal form (with respect to $\overset{\sigma\pi\alpha}{\ri\ri}$)};\\
 (\Gamma\vdash M) \boldsymbol{\Rightarrow}  U;\\
 (\Sigma\vdash L) \boldsymbol{\Rightarrow}  V;
\end{array}$\\[5pt]
then $V$ is  a $\sigma$-normal form and
$U\overset{\sigma}{\twoh}V$.
\end{theorem}
\begin{proof}$\Gamma\vdash M$ is derivable by
Lemma~\ref{poleznaya}. $\Sigma\vdash L$ is derivable by Subject
reduction. Therefore $L$ is pure  and $V$ is a $\sigma$-normal
form (Theorem~\ref{pure}, Lemma~\ref{purelemma}). Why
$U\overset{\sigma}{\twoh}V$? It is sufficient to prove that
$\sigma(U)\equiv\sigma(V)$ ($U$ reduces to its $\sigma$-normal
form because $\sigma$ is strongly normalizing and confluent). The
proof is by induction over the length of the reduction sequence $
\Gamma\vdash M\overset{\sigma\pi\alpha}{\ri\ri}\Sigma\vdash L$. If
this length is equal to $0$, there is nothing to prove. Otherwise,
suppose this sequence has the form
\\
$\Gamma\vdash M\ri\ldots\ri\Delta\vdash N\ri\Sigma\vdash L$,\\
where
 $(\Delta\vdash N)\boldsymbol{\Rightarrow}W$ and $\sigma(U)\equiv\sigma(W)$.\\
   Any possible $\sigma\pi\alpha$-reduction  step $\Delta\vdash N\ri\Sigma\vdash L$,
except  $New,\pi_1,\pi_2,\alpha_1$, and $\alpha_2$ corresponds to
the same name $\sigma$-reduction step of the  nameless terms
$W\arrow V$, hence $\sigma(W)\equiv\sigma(V)$ in these cases. For
$\alpha_1$ and $\alpha_2$ use Lemma~\ref{lemmauparrow}.\\
(\textbf{Case} $\boldsymbol{\pi_1}.) \quad \pi_{\x}\ci \y\arrow \y
\qquad(\x\neq \y)$
$$\ruletwo{\Gamma,\x\vdash
\pi_{\x}\tri\Gamma}{\ruledot{\Gamma\vdash
\y}}{\Gamma,\x\vdash\pi_{\x}\ci \y}\qquad
\ruleone{\ruledot{\Gamma\vdash \y}}{\Gamma,\x\vdash \y}$$
Suppose \\
 $(\Gamma\vdash \y)\boldsymbol{\Rightarrow}\underline{n}$;\\
  then\\
$(\Gamma,\x\vdash\pi_{\x}\ci
\y)\boldsymbol{\Rightarrow}\pi\ci\underline{n}$ and\\
$(\Gamma,\x\vdash \y)\boldsymbol{\Rightarrow}\underline{n+1}$.\\
We see that
$\sigma(\pi\ci\underline{n})\equiv\sigma(\underline{n+1})$.
\\[5pt]
(\textbf{Case} $\boldsymbol{\pi_2}.) \quad (s\ci\pi_{\x})\ci
\y\arrow s\ci \y \qquad(\x\neq \y)$
$$\ruletwo{\ruletwo{\ruledot{\Delta\vdash
s\tri\Gamma,\x}}{\Gamma,\x\vdash\pi_{\x}\tri\Gamma}{\Delta\vdash
s\ci\pi_{\x}\tri\Gamma}}{\ruledot{\Gamma\vdash \y}}{\Delta\vdash
(s\ci\pi_{\x})\ci \y}\qquad \ruletwo{\ruledot{\Delta\vdash
s\tri\Gamma,\x}}{\ruleone{\ruledot{\Gamma\vdash
\y}}{\Gamma,\x\vdash \y}}{\Delta\vdash s\ci \y}$$
Suppose \\
$(\Delta\vdash s\tri\Gamma,\x)\boldsymbol{\Rightarrow}
u$;\\
 $(\Gamma\vdash
\y)\boldsymbol{\Rightarrow}\underline{n}$;\\
then\\
$(\Delta\vdash (s\ci\pi_{\x})\ci \y)\boldsymbol{\Rightarrow} (u\ci\pi)\ci\underline{n}$;\\
$(\Gamma,\x\vdash \y)\boldsymbol{\Rightarrow} \underline{n+1}$;\\
$(\Delta\vdash s\ci \y)\boldsymbol{\Rightarrow}
u\ci\underline{n+1}$.\\
We see that\\
$\sigma((u\ci\pi)\ci\underline{n})\equiv\sigma(u\ci\underline{n+1})$\\[5pt]
(\textbf{Case} ${\mathbf{New}}.) \quad \la s\mi N\is \x\ra\ci
\y\arrow s\ci \y \qquad(\x\neq \y)$
$$\ruletwo{\ruletwo{\ruledot{\Delta\vdash
s\tri\Gamma}}{\ruledot{\Delta\vdash N}}{\Delta\vdash\la s\mi N\is
\x\ra\tri\Gamma,\x}}{\ruleone{\ruledot{\Gamma\vdash
\y}}{\Gamma,\x\vdash \y}}{\Delta\vdash\la s\mi N\is \x\ra\ci
\y}\qquad \ruletwo{\ruledot{\Delta\vdash
s\tri\Gamma}}{\ruledot{\Gamma\vdash \y}}{\Delta\vdash s\ci \y}$$
Suppose \\
$(\Delta\vdash s\tri\Gamma)\boldsymbol{\Rightarrow}
u$;\\
$(\Delta\vdash N)\boldsymbol{\Rightarrow} V$;\\
$(\Gamma\vdash \y)\boldsymbol{\Rightarrow}\underline{n}$;\\
then\\
$(\Gamma,\x\vdash \y)\boldsymbol{\Rightarrow}\underline{n+1}$;\\
$(\Delta\vdash\la s\mi N\is \x\ra\ci
\y)\boldsymbol{\Rightarrow}\la
u\mi V\ra\ci\underline{n+1}$;\\
$(\Delta\vdash s\ci \y)\boldsymbol{\Rightarrow}
u\ci\underline{n}$.\\
We see that\\
$\sigma(\la u\mi V\ra\ci\underline{n+1})\equiv\sigma(\la u\mi
V\ra\ci\pi\ci\underline{n})\equiv\sigma(u\ci\underline{n})$.
\end{proof}
\begin{theorem}\label{SN}$\sigma\pi\alpha$ is strongly
normalizing (on the sets of terms, substitutions, and judgements
of the form $\Gamma\vdash M$).
\end{theorem}
\begin{proof}The proof is postponed until Section~\ref{section9}.
\end{proof}
\begin{definition}(One-step $\beta$-reduction on the set of pure name-free terms).\\[5pt]
\ruleone{U\arrow_{Beta}V}{U\arrow_{\beta}\sigma(V)
\using{\quad(U \text{ is pure})}}\\[5pt]
(the compatible closure, of course).
\end{definition}
\begin{lemma}\label{lemmapodjom}If $U\arrow_{Beta}V$, then
$\sigma(U)\arrow^*_{\beta}\sigma(V)$, where $\arrow^*_{\beta}$ is
the reflexive closure of $\,\arrow_{\beta}$.
\end{lemma}
\begin{proof}See~\cite{Abadi}, Lemma 3.5.
\end{proof}
\begin{theorem}\label{theorem33}Suppose \\
$\begin{array}{l}
 \Gamma\vdash M\overset{\lambda\pi}{\ri\ri}\Sigma\vdash
L;\\
 (\Gamma\vdash M)\boldsymbol{\Rightarrow} U;\\
 (\Sigma\vdash L)\boldsymbol{\Rightarrow} V;
\end{array}$\\[5pt]
then $\sigma(U)\overset{\lambda\sigma}{\twoh}\sigma(V)$.
\end{theorem}
\begin{proof}The proof is by induction over the length of the reduction sequence $\Gamma\vdash M\overset{\lambda\pi}{\ri\ri}\Sigma\vdash
L$.  If this length is equal to $0$, there is nothing to prove.
Otherwise, suppose this sequence has the form
\\
$\Gamma\vdash M\ri\ldots\ri\Delta\vdash N\ri\Sigma\vdash L$,\\
where
$(\Delta\vdash N)\boldsymbol{\Rightarrow}  W$ and  $\sigma(U)\overset{\lambda\sigma}{\twoh}\sigma(W)$.\\
If the reduction step $\Delta\vdash N\ri\Sigma\vdash L$ belongs to
$\sigma\pi\alpha$, everything is all right, because
$\sigma(W)\equiv\sigma(V)$ in this case. Indeed, $\Gamma\vdash M$
is derivable by Lemma~\ref{poleznaya}. $\Delta\vdash N$ and
$\Sigma\vdash L$ are derivable by Subject
reduction. Take any $\sigma\pi\alpha$-normal form of $\Sigma\vdash L$ (this normal form exists by Theorem ~\ref{SN} and are derivable too) and use Theorem~\ref{maintheorem} to get $ W\overset{\sigma}{\twoh}\sigma(V)$.\\
If $\Delta\vdash N\ri_{Beta}\Sigma\vdash L$, then
$W\arrow_{Beta}V$, because any $Beta$-redex in $N$ corresponds to
some $Beta$-redex in $W$, hence
$\sigma(W)\arrow^*_{\beta}\sigma(V)$ by Lemma~\ref{lemmapodjom}.
\end{proof}
\begin{theorem}\label{theorem44}Suppose \\
$\begin{array}{l}
 (\Gamma\vdash M) \boldsymbol{\Rightarrow}  U;\\
 U\overset{\lambda\sigma}{\twoh}V;
\end{array}$\\
then there is $\Sigma\vdash L$ such that\\
$\begin{array}{l}
 \Gamma\vdash M\overset{\lambda\pi}{\ri\ri}\Sigma\vdash L;\\
 \Sigma\vdash L \text{ is a } \sigma\pi\alpha\text{-normal
form (with respect to $\overset{\sigma\pi\alpha}{\ri\ri}$)};\\
  (\Sigma\vdash
L)\boldsymbol{\Rightarrow}\sigma(V).
 \end{array}$
 \end{theorem}
\begin{proof}The proof is by induction over the length of the reduction
sequence $U\overset{\lambda\sigma}{\twoh}V$. \\[5pt]
Case 1: If this length is equal to $0$, take any
$\sigma\pi\alpha$-normal form of $\Gamma\vdash M$ as $\Sigma\vdash
L$ and use Theorem~\ref{maintheorem}.
\\[5pt]
 Case 2: Suppose this sequence has the form
 $U\arrow \ldots\arrow W\arrow V$ and
   the sequence
 $U\arrow \ldots\arrow W$ satisfies the statement of the theorem, i.e.:\\
 $\Gamma\vdash M
\overset{\lambda\pi}{\ri\ri}\Delta\vdash N$ for some $\Delta\vdash
N$, where $\Delta\vdash N$  is
 a
$\sigma\pi\alpha$-normal form  and $(\Delta\vdash
N)\boldsymbol{\Rightarrow}\sigma(W)$.\\[3pt]
 If the reduction step $W\arrow
V$ belongs to  $\sigma$, everything is all right,
 because
$\sigma(W)\equiv\sigma(V)$ in this case and we can use
$\Delta\vdash N$ as $\Sigma\vdash L$.
\\[3pt] If
$W\arrow_{Beta}V$, then $\sigma(W)\arrow^*_{\beta}\sigma(V)$ by
Lemma~\ref{lemmapodjom}. If $\sigma(W)$ coincides with
$\sigma(V)$, everything is all right. Otherwise, suppose
$\sigma(W)\arrow^*_{\beta}\sigma(V)$ has the form
$\sigma(W)\arrow_{Beta}W'\overset{\sigma}{\twoh}\sigma(V)$.
 Any  $Beta$-redex in
$\sigma(W)$
corresponds to some $Beta$-redex in $N$. Contracting this redex in $N$, we obtain\\
$\Delta\vdash N\ri_{Beta}\Delta\vdash N'$ and $(\Delta\vdash
N')\boldsymbol{\Rightarrow}W'$, for some $N'$. Take  any
$\sigma\pi\alpha$-normal form of $\Delta\vdash N'$ as
$\Sigma\vdash L$, then use Theorem~\ref{maintheorem} to obtain\\
$(\Sigma\vdash L)\boldsymbol{\Rightarrow}\sigma(V)$.
\end{proof}
\newpage

%% file: alpha.tex
\section{$\alpha$-equivalence}
\begin{definition} Only in this section, we use the following notation: the symbols $U,V,W$ range over \emph{extended
name-free terms} and the symbols $u,v,w$ range over \emph{extended
name-free substitutions}. The sets of extended name-free terms and
extended name-free substitutions are defined inductively as
follows:
\begin{align*}
U,V::&= \underline{n} \mid UV \mid \lambda U \mid u\ci U \\
u,v::&= id \mid \pi \mid \langle u\mi V \rangle \mid u\ci v
\end{align*}
($n\in N, n\geqslant 1$)
\end{definition}
\begin{example}$\lambda\lambda\underline{2}(\pi\ci\underline{1})$
is an extended nameless term.
\end{example}
\begin{definition}An extended nameless term $U$ is called \emph{pure} iff it does
not contain sub-terms of the form $u\ci U$.
\end{definition}
It is clear that any pure term is constructed from the symbols
$\underline{n}$ by  using application and  abstraction.
\begin{definition}A \emph{name-free judgement} is an expression of the
form $m\vdash U$ or of the form $m\vdash u$, where $m\in
N,m\geqslant 0$.
\end{definition}
Informally, $m$ is ``the length of an invisible context".
\begin{definition} $(\Gamma\vdash
M)\Rightarrow(m\vdash U)$ is shorthand for ``the name-free
judgement $m\vdash U$
 corresponds to the  judgement $\Gamma\vdash M$." \\
$(\Gamma\vdash s\tri\Delta)\Rightarrow(m\vdash u)$ is shorthand
for ``the
 name-free judgement $m\vdash u$  corresponds to the
judgement $\Gamma\vdash s\tri\Delta$."
\end{definition}
\begin{definition}By $|\Gamma|$ denote the length of $\Gamma$.
\end{definition}
\begin{definition}\label{deftransfer}(The rules of correspondence between judgements and name-free judgements).\\
$
\begin{array}{ll}
(i)& (\Gamma,\x\vdash \x)\Rightarrow  (|\Gamma,\x|\vdash \underline{1}) \\[5pt]
 (ii)& \ruleone{(\Gamma\vdash
\x)\Rightarrow  (|\Gamma| \vdash \underline{n})}{(\Gamma,\y\vdash \x)\Rightarrow  (|\Gamma,\y| \vdash \underline{n+1})\using {\quad(\x\neq \y)}}\\[20pt]
 (iii)& \ruletwo{(\Gamma\vdash M)\Rightarrow  (|\Gamma| \vdash U)}{(\Gamma\vdash N)\Rightarrow  (|\Gamma| \vdash V)}{(\Gamma\vdash
MN)\Rightarrow  (|\Gamma|  \vdash UV)}\\[15pt]
 (iv)& \ruleone{(\Gamma,\x\vdash M)\Rightarrow  (|\Gamma,\x| \vdash U)}{(\Gamma\vdash\lambda \x.M)\Rightarrow  (|\Gamma| \vdash \lambda U) }\\[15pt]
 (v)& \ruletwo{(\Gamma\vdash s \tri\Delta)\Rightarrow  (|\Gamma| \vdash u) }{(\Delta\vdash M)\Rightarrow  (|\Delta| \vdash U)}{(\Gamma\vdash
 s \ci M)\Rightarrow  (|\Gamma| \vdash u \ci U)}\\[15pt]
(vi)& (\Gamma\vdash id \tri\Gamma)\Rightarrow (|\Gamma| \vdash id ) \\[5pt]
(vii)& (\Gamma,\x\vdash\pi_{\x} \tri\Gamma)\Rightarrow  (|\Gamma,\x| \vdash \pi ) \\[5pt]
(viii)& \ruletwo{(\Gamma\vdash s \tri\Delta)\Rightarrow  (|\Gamma|
\vdash u ) }{(\Gamma\vdash N)\Rightarrow  (|\Gamma| \vdash
V)}{(\Gamma\vdash
\langle s\mi N\is \x\rangle \tri\Delta ,\x)\Rightarrow (|\Gamma| \vdash\langle u\mi V\rangle )}\\[15pt]
(ix)& \ruletwo{(\Gamma\vdash s \tri\Delta)\Rightarrow  (|\Gamma|
\vdash u) }{(\Delta\vdash q \tri\Sigma)\Rightarrow (|\Delta|
\vdash v ) }{(\Gamma \vdash s \ci q \tri\Sigma)\Rightarrow
(|\Gamma|
\vdash u \ci v )}\\
\end{array}
$
\end{definition}
\begin{corollary}If $(\Gamma\vdash M)\Rightarrow(m\vdash U)$, then
$m=|\Gamma|$.\\
 If $(\Gamma\vdash s\tri\Delta)\Rightarrow(m\vdash u)$, then $m=|\Gamma|$.
\end{corollary}
\begin{example}
$$
(x\vdash x)\Rightarrow (1 \vdash \underline{1})
$$
\end{example}
\begin{example}
$$
\ruleone{(x\vdash x)\Rightarrow (1 \vdash \underline{1})}{(x,y
\vdash x)\Rightarrow (2 \vdash \underline{2})}
$$
\end{example}
\begin{example}
$$
\ruletwo{(x,y \vdash \pi_y\tri x) \Rightarrow (2 \vdash \pi)}{(x
\vdash x) \Rightarrow (1 \vdash \underline{1})}{(x,y \vdash
\pi_y\ci x) \Rightarrow
 (2 \vdash \pi\ci \underline{1})}
$$
\end{example}
\begin{definition}($\alpha$-equivalence).\\
 We say that $\Gamma\vdash M$ is $\alpha$-\emph{equal} to $\Delta\vdash N$  and write\\
$(\Gamma\vdash M)\alp(\Delta\vdash N)$ iff\\
$(\Gamma\vdash M)\Rightarrow (m\vdash U)$ and\\ $(\Delta\vdash
N)\Rightarrow (m\vdash U)$, for some $m,U$.
\end{definition}
\begin{example}$ $\\
$(x,y\vdash\pi_y\ci x)\Rightarrow (2\vdash
\pi\ci \underline{1})$\\
$(x,x\vdash\pi_x\ci x)\Rightarrow (2\vdash
\pi\ci \underline{1})$\\
 $(x,y\vdash\pi_y\ci
x)\alp (x,x\vdash\pi_x\ci x)$
\end{example}
\begin{example}$ $\\
$(x\vdash\lambda y.\pi_y\ci x)\Rightarrow (1\vdash
\lambda\pi\ci \underline{1})$\\
$(x\vdash\lambda x.\pi_x\ci x)\Rightarrow (1\vdash
\lambda\pi\ci \underline{1})$\\
 $(x\vdash\lambda y.\pi_y\ci
x)\alp (x\vdash\lambda x.\pi_x\ci x)$
\end{example}
\begin{example}$ $\\
$(\vdash\lambda x.\lambda y.\pi_y\ci x)\Rightarrow (0\vdash
\lambda\lambda\pi\ci \underline{1})$\\
$(\vdash\lambda x.\lambda x.\pi_x\ci x)\Rightarrow (0\vdash
\lambda\lambda\pi\ci \underline{1})$\\
 $(\vdash\lambda x.\lambda y.\pi_y\ci
x)\alp (\vdash\lambda x.\lambda x.\pi_x\ci x)$
\end{example}
Warning! We can apply $\pi_1$ to the term $\lambda x.\lambda
y.\pi_y\ci x$, but not to the term $\lambda x.\lambda x.\pi_x\ci
x$ (and we can apply $\alpha_1$ to the term $\lambda x.\lambda
x.\pi_x\ci x$, but not to the term $\lambda x.\lambda y.\pi_y\ci
x$).
\begin{example}\label{counterexample}$ $\\
$(x,y\vdash x)
\backsimeq (x,y\vdash\pi_y\ci x)$\\
 $(x,y\vdash x)\,{\not\alp} (x,y\vdash\pi_y\ci x)$
 \end{example}
 \begin{lemma}If $(\Gamma\vdash M)\Rightarrow(m\vdash U)$, then
 $\Gamma\vdash M$ is derivable.
 \end{lemma}
 \begin{proof}The proof is straightforward, see Definition~\ref{defsequentsder} and
 Definition~\ref{deftransfer}.
 \end{proof}
\begin{corollary}If $(\Gamma\vdash M)\alp(\Delta\vdash N)$, then $\Gamma\vdash
M$ and $\Delta\vdash N$ are derivable.
\end{corollary}
\newpage

%% file: confluence.tex
\section{Confluence}
\begin{lemma}\label{alpequiv}$ $\\
Suppose $(\Gamma\vdash M)\alp(\Delta\vdash N)$;
 then $(\Gamma\vdash M)\backsimeq(\Delta\vdash
N)$.\\
 Suppose $(\Gamma\vdash M)\backsimeq(\Delta\vdash N)$, where
$|\Gamma|=|\Delta|$ and both $M$ and $N$ are pure;
 then $(\Gamma\vdash
M)\alp(\Delta\vdash N)$.
\end{lemma}
\begin{proof}The proof of the first part is straightforward, see Definition~\ref{transfer} and Definition~\ref{deftransfer}.  To prove the second part, recall that each pure term is constructed from variables by using
 application and abstraction. This prevents  such
 counterexamples as Example~\ref{counterexample}.
\end{proof}
\begin{theorem}[$\sigma\pi\alpha$ is confluent]\label{confluence1}
Suppose\\
$\begin{array}{l}
 (\Gamma_1\vdash M_1)\alp(\Gamma_2\vdash M_2);\\
 \Gamma_1\vdash M_1\overset{\sigma\pi\alpha}{\ri\ri}\Delta_1\vdash N_1;\\
 \Gamma_2\vdash
M_2\overset{\sigma\pi\alpha}{\ri\ri}\Delta_2\vdash N_2;
\end{array}$\\
then there are $\Sigma_1\vdash L_1$ and $\Sigma_2\vdash L_2$ such
that\\
$\begin{array}{l} \Delta_1\vdash
N_1\overset{\sigma\pi\alpha}{\ri\ri}\Sigma_1\vdash
L_1;\\
 \Delta_2\vdash N_2\overset{\sigma\pi\alpha}{\ri\ri}\Sigma_2\vdash L_2;\\
 (\Sigma_1\vdash L_1)\alp(\Sigma_2\vdash L_2).
\end{array}$
\end{theorem}
\begin{proof}By Lemma~\ref{alpequiv}, we have $(\Gamma_1\vdash M_1)\backsimeq(\Gamma_2\vdash M_2)$.\\ Suppose
$(\Gamma_1\vdash M_1)\boldsymbol{\Rightarrow}U$ and
$(\Gamma_2\vdash M_2)\boldsymbol{\Rightarrow}U$.\\
 Let $\Sigma_1\vdash L_1$ be any $\sigma\pi\alpha$-normal form of $\Delta_1\vdash N_1$ and let $\Sigma_2\vdash L_2$ be
any $\sigma\pi\alpha$-normal form of $\Delta_2\vdash N_2$. By
Theorem~\ref{maintheorem}, we have\\ $(\Sigma_1\vdash
L_1)\boldsymbol{\Rightarrow}\sigma(U)$ and  $(\Sigma_2\vdash
L_2)\boldsymbol{\Rightarrow}\sigma(U)$, hence\\ $(\Sigma_1\vdash
L_1)\backsimeq(\Sigma_2\vdash L_2)$. Note that $L_1$ and $L_2$ are
pure (Theorem~\ref{pure}). Note that
$|\Sigma_1|=|\Sigma_2|=|\Gamma_1|=|\Gamma_2|$ (because all
reductions preserve lengths of contexts). By Lemma~\ref{alpequiv},
we have\\
 $(\Sigma_1\vdash L_1)\alp(\Sigma_2\vdash L_2)$.
\end{proof}
\begin{theorem}[$\lambda\pi$ is confluent]
Suppose\\
$\begin{array}{l}
 (\Gamma_1\vdash M_1)\alp(\Gamma_2\vdash M_2);\\
 \Gamma_1\vdash M_1\overset{\lambda\pi}{\ri\ri}\Delta_1\vdash N_1;\\
 \Gamma_2\vdash
M_2\overset{\lambda\pi}{\ri\ri}\Delta_2\vdash N_2;
\end{array}$\\
then there are $\Sigma_1\vdash L_1$ and $\Sigma_2\vdash L_2$ such
that\\
$\begin{array}{l}  \Delta_1\vdash
N_1\overset{\lambda\pi}{\ri\ri}\Sigma_1\vdash
L_1;\\
 \Delta_2\vdash N_2\overset{\lambda\pi}{\ri\ri}\Sigma_2\vdash L_2;\\
 (\Sigma_1\vdash L_1)\alp(\Sigma_2\vdash L_2).
\end{array}$
\end{theorem}
\begin{proof}By Lemma~\ref{alpequiv}, we have $(\Gamma_1\vdash M_1)\backsimeq(\Gamma_2\vdash M_2)$.\\ Suppose
$(\Gamma_1\vdash M_1)\boldsymbol{\Rightarrow}U$;
$(\Gamma_2\vdash M_2)\boldsymbol{\Rightarrow}U$;\\
$(\Delta_1\vdash N_1)\boldsymbol{\Rightarrow}V_1$; and
$(\Delta_2\vdash N_2)\boldsymbol{\Rightarrow}V_2$.\\
 By
Theorem~\ref{theorem33}, we have
$\sigma(U)\overset{\lambda\sigma}{\twoh} \sigma(V_1)$ and
$\sigma(U)\overset{\lambda\sigma}{\twoh} \sigma(V_2)$.\\  We know
that $\lambda\sigma$ is confluent, hence
$\sigma(V_1)\overset{\lambda\sigma}{\twoh} V$ and
 $\sigma(V_2)\overset{\lambda\sigma}{\twoh} V$ for
some $V$. Therefore $V_1\overset{\lambda\sigma}{\twoh} V$ and
$V_2\overset{\lambda\sigma}{\twoh} V$. By Theorem~\ref{theorem44},
we have $\Sigma_1\vdash L_1$ and
$\Sigma_2\vdash L_2$ such that\\
$\begin{array}{l}  \Delta_1\vdash
N_1\overset{\lambda\pi}{\ri\ri}\Sigma_1\vdash
L_1;\\
 \Delta_2\vdash N_2\overset{\lambda\pi}{\ri\ri}\Sigma_2\vdash L_2;\\
 \Sigma_1\vdash L_1 \text{ is a $\sigma\pi\alpha$-normal form};\\
\Sigma_2\vdash L_2 \text{ is a $\sigma\pi\alpha$-normal form};\\
(\Sigma_1\vdash V_1)\boldsymbol{\Rightarrow}\sigma(V);\\
(\Sigma_2\vdash V_2)\boldsymbol{\Rightarrow}\sigma(V).
\end{array}$\\[5pt]
Hence  $(\Sigma_1\vdash L_1)\backsimeq(\Sigma_2\vdash L_2)$. Note
that $L_1$ and $L_2$ are pure (Theorem~\ref{pure}). Note that
$|\Sigma_1|=|\Sigma_2|=|\Gamma_1|=|\Gamma_2|$ (because all
reductions preserve lengths of contexts). By Lemma~\ref{alpequiv},
we have\\
 $(\Sigma_1\vdash L_1)\alp(\Sigma_2\vdash L_2)$.
\end{proof}
\begin{definition}By $\Lambda\pi$ denote the set of derivable judgements of
the form $\Gamma\vdash M$.
\end{definition}
We see that $\overset{\lambda\pi}{\ri\ri}$ and
$\overset{\sigma\pi\alpha}{\ri\ri}$ are confluent (up to $\alp$)
on the set $\Lambda\pi$.
\newpage

%% file: SN.tex
\section{$\sigma\pi\alpha$ is strongly normalizing}\label{section9}
 \begin{definition}
  $\mathcal{A} \sqsubseteq \mathcal{B}$ is shorthand for
``$\mathcal{A}_i\subseteq\bigcup_{j\geqslant i}\mathcal{B}_j$ for
all $i\geqslant 1$".
\end{definition}
\begin{example}\label{examplepi}$\langle\{y\},\emptyset,\emptyset,\ldots\rangle\sqsubseteq\langle\emptyset,\{y\},\emptyset,\emptyset,\ldots\rangle$
\end{example}
Note that $\mathcal{A} \subseteq \mathcal{B}$ implies $\mathcal{A}
\sqsubseteq \mathcal{B}$.
\begin{lemma}$O_{\lambda \x}$ and $O_{\pi}$ are monotone
operators with respect to $\sqsubseteq$ (for any $\x$).
\end{lemma}
\begin{proof}The proof is straightforward.
\end{proof}
\begin{lemma}$\mathcal{A}\cup \mathcal{B}$ is monotone in both arguments with respect
to $\sqsubseteq$.
\end{lemma}
\begin{proof}The proof is straightforward.
\end{proof}
\begin{corollary}$O_s$ is monotone  with respect to $\sqsubseteq$ for any $s$.
\end{corollary}
Recall that $O_s$ is also monotone   with respect to $\subseteq$
for any $s$.
\begin{lemma}\label{subfreevar}
If $M_1\arrow M_2$, then $FV(M_2)\sqsubseteq FV(M_1)$. If
$s_1\arrow s_2$, then $O_{s_2}(\mathcal{A})\sqsubseteq
O_{s_1}(\mathcal{A})$ for any $\mathcal{A}$.
\end{lemma}
\begin{proof}The proof is straightforward, but tedious. For example, consider\\
$
\begin{array}{lll}
\mathbf{(Abs)} & s\ci \lambda \x.M\rightarrow\lambda \x.\langle
\pi_{\x}\ci s\mi \x\is \x\rangle\ci M &
\end{array}
$\\
$\begin{array}{ll} FV(\lambda \x.\langle \pi_{\x}\ci s\mi \x\is
\x\rangle\ci M) &
\\
=FV(\lambda \x.((\pi_{\x}\ci
s)\ci\lambda \x.M)\x) & (Lemma~\ref{corollary3})\\
=FV(\lambda \x.(\pi_{\x}\ci
s\ci\lambda \x.M)\x) & (Lemma~\ref{corollary23})\\
=O_{\lambda \x}(O_{\pi}(FV
(s\ci\lambda \x.M))\bigcup FV(\x)) & \\
=O_{\lambda \x}(\la\{\x\},FV_1(s\ci\lambda \x.M),FV_2(s\ci\lambda
\x.M),\ldots\ra) & \\
=FV(s\ci\lambda \x.M) \end{array}$\\
 $\begin{array}{lll}
\mathbf{(App)} & s\ci MN \rightarrow(s\ci  M)(s\ci  N) &
\end{array}
$\\
 $\begin{array}{l}
FV(s\ci  MN) \\
=O_s(FV(MN))\\
=O_s(FV(M)\cup FV(N))\\
\supseteq
O_s(FV(M))\cup O_s(FV(N))\\
=FV(s\ci M)\cup FV(s\ci N)\\
=FV((s\ci M)(s\ci N)) \end{array}$\\
 $
\begin{array}{lll}
\mathbf{(ConsVar)} & \langle s\mi N\is \x\rangle\ci \x\rightarrow
N &
\end{array}
$\\
$\begin{array}{l}
 FV(\langle s\mi N\is \x\rangle\ci \x)\\
 =O_{\langle
s\mi N\is \x\rangle}(FV(\x))\\
=O_s(O_{\lambda \x}(FV(\x)))\cup FV(N)\\
\supseteq
FV(N)
\end{array}$\\
$
\begin{array}{lll}
\mathbf{(New)} & \langle s\mi N\is \x\rangle\ci \y\rightarrow s\ci
\y & (\x\neq \y)
\end{array}
$\\
$\begin{array}{l} FV(\langle s\mi N\is \x\rangle\ci
\y)\\
=O_{\langle s\mi N\is \x\rangle}(FV(\y))\\
=O_s(O_{\lambda \x}(FV(\y)))\cup FV(N)\\
\supseteq O_s(O_{\lambda \x}(FV(\y)))\\
=O_s(O_{\lambda
\x}(\langle\{\y\},\emptyset,\emptyset,\ldots\rangle)\\
=O_s(\langle\{\y\},\emptyset,\emptyset,\ldots\rangle)\\
=O_s(FV(\y))\\
=FV(s\ci \y)
\end{array}$\\ 
$\begin{array}{lll} \mathbf{(ConsShift)} & \langle s\mi N\is
\x\rangle\ci \pi_{\x}\rightarrow s &
\end{array}
$\\
$\begin{array}{l} O_{\langle s\mi N\is \x\rangle\ci
\pi_{\x}}(\mathcal{A})\\
=O_s(O_{\lambda \x}(O_{\pi}(\mathcal{A})))\cup FV(N)\\
\supseteq O_s(O_{\lambda
\x}(O_{\pi}(\mathcal{A})))\\
=O_s(O_{\lambda
\x}(\langle\emptyset,\mathcal{A}_1,\mathcal{A}_2,\ldots\rangle)\\
=O_s(\langle\mathcal{A}_1,\mathcal{A}_2,\ldots\rangle)\\
=O_s(\mathcal{A})
\end{array}$\\
$
\begin{array}{lll}
\mathbf{(Map)} & s\ci \langle q\mi N\is
\x\rangle\rightarrow\langle s\ci q \mi s\ci N\is \x\rangle &
\end{array}
$\\
$\begin{array}{l} O_{s\,\ci \langle q\mi N\is
\x\rangle}(\mathcal{A})\\
=O_s(O_{\langle q\mi N\is
\x\rangle}(\mathcal{A}))\\
=O_s(O_q(O_{\lambda \x}(\mathcal{A}))\cup FV(N))\\
\supseteq
O_s(O_q(O_{\lambda \x}(\mathcal{A})))\cup O_s(FV(N))\\
=O_{s\ci q}(O_{\lambda \x}(\mathcal{A}))\cup FV(s\ci
N)\\
=O_{\langle s\ci q \mi s\ci
N\is \x\rangle}(\mathcal{A})
\end{array}$\\
$
\begin{array}{lll}
\boldsymbol{(\pi_1)} & \pi_{\x}\ci \y\rightarrow \y & (\x\neq \y)
\end{array}
$\\
$\begin{array}{l}
FV(\pi_{\x}\ci
\y)\\
=\langle\emptyset,\{\y\},\emptyset,\emptyset,\ldots\rangle\\
\sqsupseteq
\langle\{\y\},\emptyset,\emptyset,\ldots\rangle\\
=FV(\y) \end{array}$
\\
$
\begin{array}{lll}
\boldsymbol{(\pi_2)} & (s\ci\pi_{\x})\ci \y\rightarrow s\ci \y &
(\x\neq \y)
\end{array}
$\\
$\begin{array}{l}
 FV((s\ci\pi_{\x})\ci
\y)\\
=O_s(O_{\pi}(FV(\y)))\\
=
O_s(\langle\emptyset,\{\y\},\emptyset,\emptyset,\ldots\rangle)\\
\sqsupseteq
O_s(\langle\{\y\},\emptyset,\emptyset,\ldots\rangle)\\
=FV(s\ci
\y)
\end{array}$\\
$\begin{array}{lll} \boldsymbol{(\alpha_1)} & \lambda
\x.M\rightarrow\lambda
\y.\langle\pi_{\y}\mi \y\is \x\rangle\ci M & (*) \end{array} $\\
$\begin{array}{ll} FV(\lambda \y.\langle\pi_{\y}\mi \y\is
\x\rangle\ci M )\\
=FV(\lambda \y.(\pi_{\y}\ci\lambda \x.M)\y) & (Lemma~\ref{corollary3})\\
=O_{\lambda \y}(O_{\pi}(FV(\lambda \x.M))\bigcup FV(\y)) & \\
=O_{\lambda \y}(\la\{\y\},FV_1(\lambda \x.M),FV_2(\lambda
\x.M),\ldots\ra) & \\
=FV(\lambda \x.M)
\end{array} $\\[3pt]
In addition, it is  necessary to prove that all operations from
Definition~\ref{pervoe} are in some sense monotone, but this is
not difficult.
\end{proof}
\begin{corollary}$FV(s\ci\lambda \x.M)=FV(\lambda
\x.\!\Uparrow_{\x}\!(s) \ci M)$
\end{corollary}
\begin{corollary}
$FV(s\ci\lambda \Delta.M)=FV(\lambda \Delta.
\!\Uparrow_{\Delta}\!(s) \ci M)$
\end{corollary}
\begin{corollary}If $M\arrow N$, then $\bigcup_{i\geqslant 1}FV_i(N)\subseteq \bigcup_{i\geqslant 1}FV_i(M)$.
\end{corollary}
 To prove that $\sigma\pi\alpha$ is strongly
normalizing, we consider the following two-sorted term rewriting
system $R$.
\begin{definition} The signature of $R$ contains:\\ $
\begin{array}{ll}
M,N,L,\ldots & \text{variables};\\
s,q,r,\ldots & \text{variables};\\
 x,y,z,\ldots & \text{constants};\\
id,\pi_x,\pi_y,\pi_z,\ldots &\text{constants};\\
\lambda x,\lambda y,\lambda z,\ldots & \text{functional symbols of
arity one};\\
\boldsymbol{\lambda} x,\boldsymbol{\lambda} y,\boldsymbol{\lambda}
z,\ldots
& \text{functional symbols of arity one};\\
\cdot\mi\circ & \text{functional symbols of arity two};\\
\la -\mi -\is x\ra,\la -\mi -\is y\ra,\la -\mi -\is z\ra,\ldots &
\text{functional symbols of arity two}.
\end{array}
$ \\[5pt]
 We
will omit $\cdot$, which denotes application. The sets of ground
terms and ground substitutions of $R$ are defined inductively as
follows:
\begin{align*}
M,N::&= \x \mid  MN \mid \lambda \x. M \mid\boldsymbol{\lambda} \x. M\mid s\ci M \\
s,q::&= id \mid \pi_{\x} \mid \langle s\mi  N\is \x \rangle \mid
s\ci q
\end{align*}
\end{definition}
 We will use the same abbreviations as in
Convention~\ref{notation2} and Convention~\ref{notation1}.
\newpage
\begin{definition}(The rewriting system $R$).\\[5pt]
$
\begin{array}{lll}
(Abs1) & s\ci \lambda \x.M\rightarrow\lambda \x.\langle
\pi_{\x}\ci s\mi \x\is
\x\rangle\ci M &\\
(Abs2) & s\ci \boldsymbol{\lambda}
\x.M\rightarrow\boldsymbol{\lambda} \x.\langle \pi_{\x}\ci s\mi
\x\is
\x\rangle\ci M &\\
(Abs3) & s\ci \lambda \x.M\rightarrow\boldsymbol{\lambda}
\x.\langle \pi_{\x}\ci s\mi \x\is
\x\rangle\ci M &\\
(Abs4) & s\ci \boldsymbol{\lambda} \x.M\rightarrow\lambda
\x.\langle \pi_{\x}\ci s\mi \x\is
\x\rangle\ci M &\\
(App) & s\ci  MN\rightarrow(s\ci  M)(s\ci  N) &\\
(ConsVar) & \langle s\mi N\is \x\rangle\ci \x\rightarrow N &\\
(New) & \langle s\mi N\is \x\rangle\ci \y\rightarrow  s\ci \y & (\x\neq \y)\\
(IdVar) &  id\ci \x\rightarrow \x &\\
(Clos) & s\ci q\ci M\rightarrow (s\ci q)\ci M &\\
(Ass) & s\ci q\ci  r\rightarrow (s\ci q)\ci r &\\
(IdR) & s\ci id\rightarrow s &\\
(IdShift) & id\ci \pi_{\x}\rightarrow\pi_{\x} &\\
(ConsShift) & \langle s\mi N\is \x\rangle\ci \pi_{\x}\rightarrow s &\\
(Map) & s\ci \langle q\mi N\is \x\rangle\rightarrow\langle
s\ci q \mi s\ci N\is \x\rangle &\\
(\pi_1) & \pi_{\x}\ci \y\rightarrow \y & (\x\neq \y)\\
(\pi_2) & (s\ci\pi_{\x})\ci \y\rightarrow s\ci \y & (\x \neq
\y)\\
(\alpha) & \boldsymbol{\lambda} \x.M\rightarrow\lambda
\y.\langle\pi_{\y}\mi \y\is \x\rangle\ci M &\\
(\xi) & \boldsymbol{\lambda} \x.M\arrow\lambda \x.M &
\end{array}
$
\end{definition}
\begin{definition} To each term $M$  we assign
$FV(M)$ as in Definition~\ref{FV} with the additional case: \\
$\begin{array}{l} FV(\boldsymbol{\lambda} \x.M)=FV(\lambda
\x.M)=O_{\lambda \x}(FV(M)) \end{array}$
\end{definition}
\begin{lemma}If $M_1\overset{R}{\arrow}M_2$, then $FV(M_2)\sqsubseteq
FV(M_1)$.
\end{lemma}
\begin{proof}See Lemma~\ref{subfreevar}.
\end{proof}
\begin{lemma}\label{restriction}The restriction $(*)$ in Definition~\ref{lambdapi} can be
written as\\
$
\begin{array}{ll}
(*) & \x \in \bigcup_{i\geqslant 1}FV_i(\lambda \x.M);\quad
\y\notin \bigcup_{i\geqslant 1}FV_i(\lambda \y.\langle\pi_{\y}\mi
\y\is \x\rangle\ci M
 )\\
 \end{array}
$
\end{lemma}
\begin{proof}$FV(\lambda \y.\langle\pi_{\y}\mi \y\is
\x\rangle\ci M )=FV(\lambda \x.M) $ by Lemma~\ref{subfreevar} (the
case $\alpha_1$).
\end{proof}
\begin{definition}By $M^*$ denote the term $M$ in which all sub-terms of
the shape $\lambda \x.L$, such that $\x\in \bigcup_{i\geqslant
1}FV_i(\lambda \x.L)$, are replaced  by $\boldsymbol{\lambda}
\x.L$.
\end{definition}
\begin{theorem}If $R$ is strongly normalizing on the sets of ground terms and ground substitutions, then $\sigma\pi\alpha$ is strongly normalizing (on the sets of terms, substitutions, and judgements of
the form $\Gamma\vdash M$).
\end{theorem}
\begin{proof}Suppose we have some infinite $\sigma\pi\alpha$-sequence
$$ M_1\arrow M_2\arrow\ldots\arrow M_n\arrow\ldots$$
 I claim that we can get
some
infinite  $R$-sequence\\
$$(M_1)^*\overset{R}{\twoh}(M_2)^*\overset{R}{\twoh}\ldots\overset{R}{\twoh}(M_n)^*\overset{R}{\twoh}\ldots$$
The proof is by induction over $n$. If $n$ is equal to $1$, there is nothing to prove. Else there are three cases.\\[5pt]
1)  If the reduction step $ M_n\arrow M_{n+1}$ is not
$Abs,\alpha_1,\alpha_2$, we can apply the\\ $R$-reduction of the
same name $(M_n)^*\arrow( M_{n+1})^*$, but then might need several
$\xi$-steps, because $ConsVar, New, ConsShift,\pi_1,\pi_2$ can
decrease $FV$. If any of these reductions is applied under some
black lambda, this lambda may
turn pale. See Example~\ref{example1.18}.\\[5pt]
2) If  $M_n\arrow_{\alpha_1} M_{n+1}$,
we can apply $\alpha$:\\
$(M_n)^*\arrow_{\alpha}(M_{n+1})^*$ (see Lemma~\ref{restriction}
and
Example~\ref{example1.13}).\\[5pt]
3) If $ M_n\arrow_{Abs} M_{n+1}$ and the $Abs$-redex is
$s\ci\lambda \x.M$, there are four
possible subcases:\\
$
\begin{array}{lll}
Subcase \,1. & \x\notin FV_1(\lambda \x.M),\, \x\notin
FV_1(\lambda
\x.\langle \pi_{\x}\ci s\mi \x\is \x\rangle\ci M);\\
Subcase\, 2. & \x\in FV_1(\lambda \x.M),\, \x\in FV_1(\lambda
\x.\langle \pi_{\x}\ci s\mi \x\is \x\rangle\ci M);\\
Subcase\, 3. & \x\notin FV_1(\lambda \x.M),\, \x\in FV_1(\lambda
\x.\langle \pi_{\x}\ci s\mi \x\is \x\rangle\ci M);\\
Subcase \,4. & \x\in FV_1(\lambda \x.M),\, \x\notin FV_1(\lambda
\x.\langle \pi_{\x}\ci s\mi \x\is \x\rangle\ci M);
\end{array}\\[3pt]
$ and we can apply $Abs_1, Abs_2, Abs_3$, and $Abs_4$,
respectively. See Examples~\ref{example1.14}, \ref{example1.15},
\ref{example1.16}, and \ref{example1.17}.

The proof is similar for substitutions. For judgements, suppose we
have some infinite $\sigma\pi\alpha$-sequence
$$\Gamma_1\vdash M_1\ri\Gamma_2\vdash  M_2\ri\Gamma_3\vdash M_3\ri\ldots$$
We can obtain the $\sigma\pi\alpha$-sequence of terms
$$\Lambda\Gamma_1.M_1\twoh\Lambda\Gamma_2.M_2\twoh\Lambda\Gamma_3.M_3\twoh\ldots$$
where $\alpha_2$-steps are replaced by $\alpha_1$ and $Abs$.
\end{proof}
\begin{example}\label{example1.18}The $\sigma\pi\alpha$-sequence
$$\lambda x.\la id\mi\pi_x\ci x\is x\ra\ci y\arrow_{New}\lambda
x.id\ci y\arrow\ldots$$
 becomes the following $R$-sequence
$$\boldsymbol{\lambda} x.\la id\mi\pi_x\ci x\is x\ra\ci
y\arrow_{New} \boldsymbol{\lambda} x.id\ci y \arrow_{\xi}\lambda
x.id\ci y\arrow\ldots$$
\end{example}
\begin{example}\label{example1.13}The $\sigma\pi\alpha$-sequence
 $$\lambda x.\pi_x\ci
 x\arrow_{\alpha_1}\lambda y.\la\pi_y\mi y\is
 x\ra\ci\pi_x\ci x\arrow\ldots$$
  becomes the following  $R$-sequence
 $$\boldsymbol{\lambda} x.\pi_x\ci x\arrow_{\alpha}\lambda y.\la\pi_y\mi y\is
 x\ra\ci\pi_x\ci x\arrow\ldots$$
 \end{example}
\begin{example}\label{example1.14}The $\sigma\pi\alpha$-sequence
$$id\ci\lambda x.x\arrow_{Abs}\lambda x.\la\pi_x\ci id\mi x\is
x\ra\ci x\arrow\ldots$$
 becomes the following  $R$-sequence
$$ id\ci\lambda x.x\arrow_{Abs1}\lambda x.\la\pi_x\ci id\mi x\is
x\ra\ci x\arrow\ldots$$
\end{example}
\begin{example}\label{example1.15}The $\sigma\pi\alpha$-sequence
$$id\ci\lambda x.\pi_x\ci x\arrow_{Abs}\lambda x.\la\pi_x\ci id\mi
x\is x\ra\ci\pi_x\ci x\arrow\ldots$$
 becomes the following
$R$-sequence
$$id\ci\boldsymbol{\lambda} x.\pi_x\ci
x\arrow_{Abs2}\boldsymbol{\lambda} x.\la\pi_x\ci id\mi x\is
x\ra\ci\pi_x\ci x\arrow\ldots$$
\end{example}
\begin{example}\label{example1.16}The $\sigma\pi\alpha$-sequence
$$\la id\mi x\is y\ra\ci\lambda x.x\arrow_{Abs}\lambda
x.\la\pi_x\ci\la id\mi x\is y\ra\mi x\is x\ra\ci x\arrow\ldots$$
becomes the following  $R$-sequence
$$\la id\mi x\is
y\ra\ci\lambda x.x\arrow_{Abs3} \boldsymbol{\lambda}
x.\la\pi_x\ci\la id\mi x\is y\ra\mi x\is x\ra\ci x\arrow\ldots$$
\end{example}
\begin{example}\label{example1.17}The $\sigma\pi\alpha$-sequence
$$\la id\mi\lambda y.y\is x\ra\ci\lambda x.\pi_x\ci
x\arrow_{Abs}\lambda x.\la\pi_x\ci\la id\mi\lambda y.y\is x\ra\mi
x\is x\ra\ci\pi_x\ci x\arrow\ldots$$
 becomes the following
$R$-sequence
$$\la id\mi\lambda y.y\is x\ra\ci\boldsymbol{\lambda}
x.\pi_x\ci x\arrow_{Abs4}\lambda x.\la\pi_x\ci\la id\mi\lambda
y.y\is x\ra\mi x\is x\ra\ci\pi_x\ci x\arrow\ldots$$
\end{example}
To prove that $R$ is strongly normalizing on the sets of ground
terms and ground substitutions, we use the method of semantic
labelling. See~\cite{Zantema}.
\begin{definition}
To each term $M$ and each substitution $s$ we put in correspondence  natural numbers  $|M|$ and $|s|$ respectively defined as follows:\\
$\begin{array}{l}
|\lambda \x.M|=|M|+1\\
|\boldsymbol{\lambda} \x.M|=|M|+1\\
|s\ci M|=|s|+|M|\\
|s\ci
q|=|s|+|q|\\
|MN|=max(|M|,|N|)\\
|\la s\mi N\is
\x\ra|=max(|s|,|N|)\\
|id|=0\\
|\pi_{\x}|=0\\
 |\x|=0
 \end{array}$
\end{definition}
Note that any functional symbol of $R$ now turns to some
 monotone function of $\mathbb{N}$ to
$\mathbb{N}$ or of $\mathbb{N}\times\mathbb{N}$ to $\mathbb{N}$.
Consider  the following two-sorted term rewriting system $Q$.
\begin{definition} The
signature
of $Q$ contains:\\
$
\begin{array}{ll}
M,N,L,\ldots & \text{variables};\\
s,q,r,\ldots & \text{variables};\\
x,y,z,\ldots & \text{constants};\\
id,\pi_x,\pi_y,\pi_z,\ldots &\text{constants};\\
\lambda x,\lambda y,\lambda z,\ldots & \text{functional symbols of
arity one};\\
\boldsymbol{\lambda}_i x,\boldsymbol{\lambda}_i
y,\boldsymbol{\lambda}_i z,\ldots
& \text{functional symbols of arity one};\\
\cdot\mi\circ_i & \text{functional symbols of arity two};\\
\la -\mi -\is x\ra,\la -\mi -\is y\ra,\la -\mi -\is z\ra,\ldots &
\text{functional symbols of arity two};
\end{array}
$ \\[5pt]
where $i\in \mathbb{N},i\geqslant 0$.\\[5pt]
 We will
omit $\cdot$, which denotes application. The sets of ground terms
and ground substitutions of $Q$ are defined inductively as
follows:
\begin{align*}
M,N::&= \x \mid  MN \mid \lambda \x. M \mid\boldsymbol{\lambda}_i \x. M\mid s\ci_i M \\
s,q::&= id \mid \pi_{\x} \mid \langle s\mi  N\is \x \rangle \mid
s\ci_i q
\end{align*}
\end{definition}
 We will use the same abbreviations as in
Convention~\ref{notation2} and Convention~\ref{notation1}.
\newpage
\begin{definition}(The rewriting system $Q$).\\[5pt]
$
\begin{array}{lll}
(Abs1) & s\ci_{i+1} \lambda \x.M\rightarrow\lambda \x.\langle
\pi_{\x}\ci_k s\mi \x\is
\x\rangle\ci_i M & (i \geqslant k)\\
(Abs2) & s\ci_{i+1} \boldsymbol{\lambda}_{j+1}
\x.M\rightarrow\boldsymbol{\lambda}_{i+1} \x.\langle \pi_{\x}\ci_k
s\mi \x\is
\x\rangle\ci_i M & (i =j+k)\\
(Abs3) & s\ci_{i+1 }\lambda
\x.M\rightarrow\boldsymbol{\lambda}_{i+1} \x.\langle \pi_{\x}\ci_k
s\mi \x\is
\x\rangle\ci_i M & (i \geqslant k)\\
(Abs4) & s\ci_{i+1} \boldsymbol{\lambda}_{j+1}
\x.M\rightarrow\lambda \x.\langle \pi_{\x}\ci_k s\mi \x\is
\x\rangle\ci_i M & (i=j+k)\\
(App) & s\ci_i MN\rightarrow(s\ci_j  M)(s\ci_k  N) & (i\geqslant j, i\geqslant k)\\
(ConsVar) & \langle s\mi N\is \x\rangle\ci_i \x\rightarrow N &\\
(New) & \langle s\mi N\is \x\rangle\ci_i \y\rightarrow  s\ci_j \y & (\x\neq \y, i\geqslant j)\\
(IdVar) &  id\ci_0 \x\rightarrow \x &\\
(Clos) & s\ci_{i+j+k} q\ci_{j+k} M\rightarrow (s\ci_{i+j} q)\ci_{i+j+k} M &\\
(Ass) & s\ci_{i+j+k} q\ci_{j+k}  r\rightarrow (s\ci_{i+j} q)\ci_{i+j+k} r &\\
(IdR) & s\ci_i id\rightarrow s &\\
(IdShift) & id\ci_0 \pi_{\x}\rightarrow\pi_{\x} &\\
(ConsShift) & \langle s\mi N\is \x\rangle\ci_i \pi_{\x}\rightarrow s &\\
(Map) & s\ci_i \langle q\mi N\is \x\rangle\rightarrow\langle
s\ci_j q \mi s\ci_k N\is \x\rangle & (i
\geqslant j, i \geqslant k)\\
(\pi_1) & \pi_{\x}\ci_0 \y\rightarrow \y & (\x\neq \y)\\
(\pi_2) & (s\ci_i\pi_{\x})\ci_i \y\rightarrow s\ci_i \y & (\x \neq
\y)\\
(\alpha) & \boldsymbol{\lambda}_{i+1} \x.M\rightarrow\lambda
\y.\langle\pi_{\y}\mi \y\is \x\rangle\ci_i M &\\
(\xi) & \boldsymbol{\lambda}_{i+1} \x.M\arrow\lambda \x.M &\\
 (Decr_1) &
\boldsymbol{\lambda}_i \x.M\arrow\boldsymbol{\lambda}_j \x.M & (i
> j)\\
(Decr_2) & s\ci_i M\arrow s\ci_j M & (i>j)\\
(Decr_3) & s\ci_i q\arrow s\ci_j q & (i>j)
\end{array}
$ \\[5pt]
where $i,j,k\in N$. (Roughly, these are the rewrite rules of $R$,
where $\ci$ and $\boldsymbol{\lambda} \x$ are labelled by theirs
own values).
\end{definition}
\newpage
\begin{theorem}\label{normalizing}$Q$ is strongly normalizing on the sets of ground
terms and ground substitutions.
\end{theorem}
\begin{proof}
By choosing the  well-founded precedence\\
$
\begin{array}{ll}
\boldsymbol{\lambda}_{i+1} \x>\ci_i>\boldsymbol{\lambda}_i
\x & \text{for all } i,\x;\\
\ci_i>\lambda \x & \text{for all } i,\x;\\
\ci_i>\cdot &\text{for all } i;\\
\ci_i>\la -\mi -\is \x\ra & \text{for all } i,\x;\\
\ci_i>\pi_{\x} & \text{for all } i,\x;\\
\ci_i> \x & \text{for all } i,\x;\\
\boldsymbol{\lambda}_i \x>\lambda \y & \text{for all } i,\x,\y;\\
\boldsymbol{\lambda}_{i} \x>\la-\mi -\is \x\ra & \text{for all } i,\x;\\
\boldsymbol{\lambda}_{i} \x>\pi_{\y} & \text{for all } i,\x,\y;\\
\boldsymbol{\lambda}_{i} \x> \y & \text{for all } i,\x,\y;\\
 \boldsymbol{\lambda}_i
\x>\boldsymbol{\lambda}_j \x & \text{for }
 i>j; \\
\ci_i>\ci_j & \text{for } i>j;
\end{array}
$\\[5pt]
termination is easily proved by the lexicographic path order.
\end{proof}
\begin{theorem}$R$ is strongly normalizing on the sets of ground
terms and ground substitutions.
\end{theorem}
\begin{proof}
For any infinite $R$-sequence
$$M_1\arrow M_2\arrow M_3\arrow\ldots$$
we can get some infinite $Q$-sequence simply by labelling all
symbols $\circ$ and $\boldsymbol{\lambda} \x$ by theirs own
values. See~\cite{Zantema}, Theorem~81 for details (see
also~\cite{Zantema}, Example~33). The proof  is similar for
substitutions.
\end{proof}
\newpage

%% file: Post.tex
\section{Post canonical system for terms and substitutions}
In this section we consider some Post canonical system building
the sets of terms and substitutions. The alphabet of the system
contains:\\[3pt]
$ \begin{array}{ll}
 x\mid y\mid z\mid \lambda\mid .\mid \ci\mid
id\mid \pi\mid \la\mid \mi\mid\ra\mid (\mid ) & symbols\\
\term\mid \sub\mid \var\mid \appterm\mid \absterm\mid
\closterm\mid \conssub\mid \compsub \mid
 \varlist\mid \slist
&
symbols\\
 a\mid b\mid M\mid N\mid s\mid q\mid G & variables
\end{array}$\\[5pt]

$\term M$ means that $M$ is a term. $\sub s$ means that $s$ is a
substitution. $\var a$ means that $a$ is a variable. $\appterm M$
means that $M$ is a term of the form $M_1M_2\ldots M_n$, where
$n\geqslant 2$. $\absterm M$ means that $M$ is a term of the form
$\lambda a_1\ldots a_n.N$, where $n\geqslant 1$. $\closterm M$
means that $M$ is a term of the form $s\ci N$. $\conssub s$ means
that $s$ is a substitution of the form $\la q\mi N_1\is
a_1,\ldots,N_n\is a_n\ra$, where $n\geqslant 1$. $\compsub s$
means that $s$ is a substitution of the form $s_1\ci
s_2\ci\ldots\ci
s_n$, where $n\geqslant 2$. We will write $\pi_a$ instead of $\pi a$.\\[5pt]

The general rules:\\[5pt]

\ruleone{\var a}{\term a}\qquad \ruleone{\appterm  M}{\term
M}\qquad \ruleone{\absterm  M}{\term M}\qquad\ruleone{\closterm
M}{\term
M}\\[10pt]

$\sub  id$\qquad\ruleone{\var  a}{\sub
\pi_a}\qquad\ruleone{\conssub s}{\sub  s}\qquad\ruleone{\compsub
s}{\sub  s}\\[10pt]

For simplicity, we will use only $x,y,z$.\\[3pt]

$\var x$\qquad$\var y$\qquad$\var z$\\[5pt]

The following rules build terms of the form\\[1pt]
$ab$,\quad $a(N_1N_2\ldots N_k)$,\quad $a(\lambda a_1\ldots a_k.N)$,\quad $a(s\ci N)$\\[5pt]

\ruletwo{\var a}{\var b}{\appterm  ab}\\[5pt]

\ruletwo{\var a}{\appterm  N}{\appterm  a(N)}\\[5pt]

\ruletwo{\var a}{\absterm  N}{\appterm  a(N)}\\[5pt]

\ruletwo{\var a}{\closterm  N}{\appterm  a(N)}\\[5pt]

The following rules build terms of the forms\\[1pt]
$M_1M_2\ldots M_n \,a$,\quad $M_1M_2\ldots M_n(N_1N_2\ldots
N_k)$,\quad $M_1M_2\ldots M_n(\lambda a_1\ldots a_k.N)$,\\
and  $M_1M_2\ldots M_n(s\ci N)$\\[5pt]

\ruletwo{\appterm  M}{\var a}{\appterm  Ma}\\[5pt]

\ruletwo{\appterm  M}{\appterm  N}{\appterm  M(N)}\\[5pt]

\ruletwo{\appterm  M}{\absterm  N}{\appterm  M(N)}\\[5pt]

\ruletwo{\appterm  M}{\closterm  N}{\appterm  M(N)}\\[5pt]

The following rules build terms of the forms\\[1pt]
$(\lambda a_1\ldots a_n.M)\,a$,\quad $(\lambda a_1\ldots
a_n.M)(N_1N_2\ldots N_k)$,\quad $(\lambda a_1\ldots a_n.M)(\lambda
b_1\ldots b_k.N)$,\\
and $(\lambda a_1 \ldots
a_n.M)(s\ci N)$\\[5pt]

\ruletwo{\absterm  M}{\var a}{\appterm  (M)a}\\[5pt]

\ruletwo{\absterm  M}{\appterm  N}{\appterm  (M)(N)}\\[5pt]

\ruletwo{\absterm  M}{\absterm  N}{\appterm (M)(N)}\\[5pt]

\ruletwo{\absterm  M}{\closterm  N}{\appterm  (M)(N)}\\[5pt]

The following rules build terms of the forms\\[1pt]
$(s\ci M)\,a$,\quad $(s\ci M)(N_1N_2\ldots N_k)$,\quad $(s\ci
M)(\lambda a_1\ldots a_k.N)$,\quad
$(s\ci M)(q\ci N)$\\[5pt]

\ruletwo{\closterm  M}{\var a}{\appterm  (M)a}\\[5pt]

\ruletwo{\closterm  M}{\appterm  N}{\appterm  (M)(N)}\\[5pt]

\ruletwo{\closterm  M}{\absterm  N}{\appterm  (M)(N)}\\[5pt]

\ruletwo{\closterm  M}{\closterm  N}{\appterm  (M)(N)}\\[5pt]

The following rules build lists of variables:\\[5pt]

\ruleone{\var a}{\varlist a}\qquad\ruletwo{\varlist G}{\var
a}{\varlist Ga}\\[5pt]

The following rules build terms of the forms\\[1pt]
$\lambda a_1\ldots a_n.a$,\quad $\lambda a_1\ldots
a_n.M_1M_2\ldots
M_k$,\quad  $\lambda a_1\ldots a_n.s\ci M$\\[5pt]

\ruletwo{\varlist G}{\var a}{\absterm  \lambda
G.a}\\[5pt]

\ruletwo{\varlist  G}{\appterm  M}{\absterm \lambda
G.M}\\[5pt]

\ruletwo{\varlist  G}{\closterm  M}{\absterm  \lambda  G.M}\\[5pt]

The following rules build terms of the forms\\[1pt]
 $id\ci M$,\quad
$\pi_a\ci M$,\quad $\la s\mi N_1\is a_1,\ldots,N_n\is a_n\ra\ci
M$,\quad
 $(s_1\ci s_2\ci\ldots\ci s_n)\ci M$\\[5pt]

\ruleone{\term  M}{\closterm  id\ci M}\\[5pt]

\ruletwo{\var a}{\term  M}{\closterm  \pi_a\ci  M}\\[5pt]

\ruletwo{\conssub  s}{\term  M}{\closterm  s\ci  M}\\[5pt]

\ruletwo{\compsub  s}{\term  M}{\closterm  (s)\ci  M}\\[5pt]

The following rules build lists of the form\\[1pt]
 $N_1\is
a_1\mi\ldots\mi N_n\is a_n$\\[5pt]

\ruletwo{\term  N}{\var a}{\slist  N\is a}\qquad \prooftree
             \slist G\quad
             \term N\quad
             \var a
      \justifies
          \slist G\mi N\is a\qquad\quad
      \thickness=0.08em
      \shiftright 2em
     \endprooftree\\[5pt]

The following rules build substitutions of the forms\\[1pt]
$\la id\mi N_1\is a_1,\ldots,N_n\is a_n\ra$,\quad $\la\pi_a\mi
N_1\is a_1,\ldots,N_n\is a_n\ra$,\\
and  $\la s_1\ci s_2\ci\ldots\ci s_k\mi N_1\is a_1,\ldots,N_n\is a_n\ra$\\[5pt]

\ruleone{\slist G}{\conssub  \la id\mi
G\ra}\\[5pt]

\ruletwo{\var a}{\slist  G}{\conssub \la\pi_a\mi G\ra}\\[5pt]

\ruletwo{\compsub  s}{\slist  G}{\conssub \la s\mi G\ra}\\[5pt]

The following rules build substitutions of the forms\\[1pt]
$id\ci id$,\quad $\pi_a\ci id$,\quad $\la s\mi N_1\is
a_1,\ldots,N_n\is a_n\ra\ci id$,\quad $(s_1\ci s_2\ci\ldots\ci
s_n)\ci id$\\[5pt]

$\compsub  id\ci id$\\[5pt]

\ruleone{\var a}{\compsub  \pi_a\ci id}\\[5pt]

\ruleone{\conssub  s}{\compsub s\ci id}\\[5pt]

\ruleone{\compsub  s}{\compsub  (s)\ci id}\\[5pt]

The following rules build substitutions of the forms\\[1pt]
$id\ci\pi_a$,\quad $\pi_b\ci\pi_a$,\quad $\la s\mi N_1\is
a_1,\ldots,N_n\is a_n\ra\ci\pi_a$,\quad $(s_1\ci
s_2\ci\ldots\ci s_n)\ci\pi_a$\\[5pt]

\ruleone{\var a}{\compsub id\ci\pi_a}\\[5pt]

\ruletwo{\var b}{\var a}{\compsub\pi_b\ci\pi_a}\\[5pt]

\ruletwo{\conssub s}{\var a}{\compsub s\ci\pi_a}\\[5pt]

\ruletwo{\compsub  s}{\var a}{\compsub
(s)\ci\pi_a}\\[5pt]

The following rules build substitutions of the forms\\[1pt]
$id\ci\la q\mi M_1\is b_1,\ldots,M_k\is b_k\ra$,\quad $\pi_a\ci\la
q\mi M_1\is
b_1,\ldots,M_k\is b_k\ra$,\\
  $\la s\mi N_1\is a_1,\ldots,N_n\is
a_n\ra\ci\la q\mi M_1\is
b_1,\ldots,M_k\is b_k\ra$,\\
and $(s_1\ci s_2\ci\ldots\ci s_n)\ci\la q\mi M_1\is
b_1,\ldots,M_k\is b_k\ra$\\[5pt]

\ruleone{\conssub q}{\compsub id\ci q}\\[5pt]

\ruletwo{\var a}{\conssub q}{\compsub \pi_a\ci q}\\[5pt]

\ruletwo{\conssub s}{\conssub q}{\compsub s\ci q}\\[5pt]

\ruletwo{\compsub  s}{\conssub  q}{\compsub (s)\ci q}\\[5pt]

The following rules build substitutions of the forms\\[1pt]
$id\ci q_1\ci q_2\ci\ldots\ci q_k$,\quad $\pi_a\ci q_1\ci
q_2\ci\ldots\ci q_k$,\\ $\la s\mi N_1\is a_1,\ldots, N_n\is
a_n\ra\ci q_1\ci q_2\ci\ldots\ci q_k$,\\
 and $(s_1\ci s_2\ci\ldots\ci
s_n)\ci q_1\ci q_2\ci\ldots\ci q_k$\\[5pt]

\ruleone{\compsub  q}{\compsub  id\ci q}\\[5pt]

\ruletwo{\var a}{\compsub  q}{\compsub  \pi_a\ci q}\\[5pt]

\ruletwo{\conssub  s}{\compsub  q}{\compsub  s\ci q}\\[5pt]

\ruletwo{\compsub  s}{\compsub  q}{\compsub (s)\ci q}\\[5pt]

\newpage

%% file: note.tex
\section{Notes}
(1) We can accept $Abs$ in the stronger form\\
$\begin{array}{lll}  & s\ci\lambda \x.M\arrow\lambda
\y.\la\pi_{\y}\ci s\mi \y\is \x\ra\ci M & (\x,\y \text{ are
arbitrary})
\end{array}$\\[5pt]
All results of this article remain true.
We can also add the following rewrite rules\\
$\begin{array}{ll} & id\ci M\arrow M\\
& id\ci s\arrow s
\end{array}$\\[5pt]
All results of this article remain true.\\

(2) It is easy to add some $\alpha_2$-like reduction for
substitutions, but that little benefit, because the analogue of
Lemma~\ref{lemmauparrow} is false for substitutions.\\

(3) It is easy to give the following definitions:
\begin{definition}(Free variables of substitutions). By definition, put\label{FV}\\
$
\begin{array}{ll}
(i) & FV(id)=\langle \emptyset,\emptyset,\emptyset,\ldots\rangle\\
(ii) & FV(\pi_{\x})=\langle \emptyset,\emptyset,\emptyset,\ldots\rangle\\
(iii) & FV(\la s\mi N\is \x\ra)=FV(s)\cup FV(N)\\
 (iv) & FV(s\ci q)=O_s(FV(q))
\end{array}
$
\end{definition}
We see that $FV(s)=O_s(\langle
\emptyset,\emptyset,\emptyset,\ldots\rangle)$

\begin{definition}($\alpha$-equivalence for substitutions).\\
 We say that $\Gamma\vdash s\tri\Delta$ is $\alpha$-\emph{equal} to $\Sigma\vdash
q\tri\Psi$ and write\\ $(\Gamma\vdash
s\tri\Delta)\alp(\Sigma\vdash
q\tri\Psi)$ iff\\
$(\Gamma\vdash s\tri\Delta)\Rightarrow (m\vdash u)$ and\\
$(\Sigma\vdash q\tri\Psi)\Rightarrow (m\vdash u)$, for some $m,u$.
\end{definition}
\begin{example}$ $\\
$(x,y\vdash\pi_y\tri x)\Rightarrow(2\vdash \pi)$\\
$(x,x\vdash\pi_x\tri x)\Rightarrow(2\vdash\pi)$\\
$(x,y\vdash\pi_y\tri x)\alp (x,x\vdash \pi_x\tri x)$
\end{example}